\documentclass[11pt]{amsart}
\usepackage{amsmath,amsfonts,amssymb,amsthm,epsfig,color, mathrsfs} 
\usepackage{hyperref} 
\usepackage{bbm}
 \usepackage{pgfplots}
  \usepackage{graphicx}
\pgfplotsset{graph default/.append style={axis x line=middle, axis y line= middle, xlabel={$u$}, ylabel={$w$}, x axis line style=-, y axis line style=-}}

\makeatletter
\@namedef{subjclassname@2020}{%
  \textup{2020} Mathematics Subject Classification}
\makeatother

\newcommand{\bb}{\mathbb}

\newcommand{\CC}{\bb C}

\newcommand{\ZZ}{\bb Z}

\newcommand{\RR}{\bb R}

\newcommand{\A}{\bb A}

\newcommand{\eL}{\mathcal{L}}

\newcommand{\Ga}{\Gamma}

\newcommand*\wrt{\mathop{}\!\mathrm{d}}
\newcommand \Wrt[1]{\mathop{}\!\mathrm{d^#1}}

\renewcommand{\Im}{\operatorname{Im}}
\renewcommand{\Re}{\operatorname{Re}}

\newtheorem{Theorem}{Theorem}
\numberwithin{Theorem}{section}

\newtheorem{theo}[Theorem]{Theorem}
\newtheorem{prop}[Theorem]{Proposition}

\newtheorem*{lemma*}{Lemma}
\newtheorem*{question*}{Question}

\newtheorem*{theorem*}{Theorem}

\theoremstyle{remark}

\newtheorem{rema}[Theorem]{\sc Remark}
\newtheorem*{rema*}{\sc Remark}
\numberwithin{equation}{section}
\numberwithin{figure}{section}
\begin{document}
\title{An asymptotic for the $K$-Bessel function using the saddle-point method}
\author{Jimmy Tseng}

\address{Department of Mathematics, University of Exeter, Exeter, EX4 4QF, UK}
\email{j.tseng@exeter.ac.uk}
 \thanks{The author was supported by EPSRC grant EP/T005130/1.}
\begin{abstract}  Using the saddle-point method, we compute an asymptotic, as $y \rightarrow \infty$, for the $K$-Bessel function $K_{r  + i t}(y)$ with positive, real argument $y$ and of large complex order $r+it$ where $r$ is bounded and $t = y \sin \theta$ for a fixed parameter $0\leq \theta\leq \pi/2$ or $t= y \cosh \mu$ for a fixed parameter $\mu>0$.  Our method gives an illustrative proof, using elementary tools, of this known result and explains how these asymptotics come about.

As part of our proof, we prove a new result, namely a novel integral representation for $K_{r  + i t}(y)$ in the case $t= y \cosh \mu$.  This integral representation involves only one saddle point.

\end{abstract}

\maketitle
\tableofcontents

\section {Introduction}\label{secBackground}

Bessel functions are solutions to certain second-order differential equations.  There are a number of variants and they have applications in physics (and other sciences), engineering, statistics, and within mathematics itself (for but a few examples, see~\cite{Boc1892, Con07, DRS10, Kor02, Rob90,SF75, Wa}).  The $K$-Bessel function (see (\ref{eqnKInTermsOfI}) for the definition) is one of these variants and it, in particular, has important applications in analytic number theory and ergodic theory, especially as it appears in the Fourier expansion of eigenforms of the non-Euclidean Laplacian, such as Hecke-Maass forms and Eisenstein series, (see~\cite[Chapter~3]{Iwa02} and~\cite{Zha10} for example).  Understanding the $K$-Bessel function and, in particular, its asymptotics will have applications in science and mathematics such as, for example, computing bounds on the Eisenstein series~\cite[Theorem~1.12]{Ts18} (see also~\cite{St}).

 In this paper, we give a proof of these asymptotics, first proved in~\cite[Theorems~1.1 and~1.3]{Ts18}.  The original proof uses Laplace's method, which is a powerful tool but does not yield much insight.  Our proof, unlike that in~\cite{Ts18}, explains, using elementary tools, how these asymptotics come about.  We will find a suitable integral representation of the $K$-Bessel function and the relevant saddle point and will show that the integral is dominated by its restriction to a small (suitable) neighborhood of the saddle point and is negligible outside of this neighborhood.  In addition, our proof also yields a novel integral representation for the $K$-Bessel function.  This integral representation is a new result and may be of independent interest.

\subsection{Statement of results}

Let $\nu := r +it$.  The asymptotics of the $K$-Bessel function that we compute using our illustrative method are stated for the two cases $y \geq t \geq 0$ and $0 < y < t$ in Theorems~\ref{thmFirstCaseMonoKBessel} and~\ref{thmSecondCaseOscillKBessel}, respectively, and were first proved in~\cite{Ts18} (\cite[Theorems~1.1~and~1.3]{Ts18}, respectively).  When $t <0$, see Remark~\ref{rmkNegativeTIsOK}.  Note that $\Gamma(\cdot)$ is the gamma function.  The definitions of the symbols $\sim$, $o(\cdot)$, and $O(\cdot)$ can be found in~\cite[Chapter 1, Section 2]{Olv}.  


\begin{theo}\label{thmFirstCaseMonoKBessel}
Let $M\geq0$ and $0\leq\theta \leq \pi/2$ be fixed real numbers.  Let $|r|\leq M$, $0< y \in \RR$, and \begin{align}
t  = y \sin \theta.\end{align}  Then \begin{align*}
K_\nu(y) \sim \begin{cases} \sqrt{\frac{\pi}{2 y \cos \theta}}e^{-y (\cos \theta +\theta \sin \theta )}e^{ir\theta} & \text{ if } 0 \leq \theta < \frac \pi 2 \\ e^{-\frac \pi 2 y +i\frac \pi 2 r}y^{-1/3}\frac{\Gamma(\frac 1 3)} {2^{\frac 2 3}3^{\frac 1 6}} & \text{ if } \theta = \frac \pi 2\end{cases}  \end{align*} as $y \rightarrow \infty$.
\end{theo}

%

\begin{theo}\label{thmSecondCaseOscillKBessel}
Let $M \geq0$ and $\mu > 0$ be fixed real numbers.  Let $|r|\leq M$, $0< y \in \RR$, and \begin{align}
t  = y \cosh \mu.\end{align}  Then \begin{align*}
K_\nu(y) = \sqrt{\frac{2\pi}{y \sinh \mu}}e^{-y \frac \pi 2 \cosh \mu +i r \frac \pi 2}& \left[\cosh(r \mu) \sin\left(\frac \pi 4 - y \left(\sinh \mu  - \mu \cosh \mu \right)\right) \right. \\ &\left.\quad  -i \sinh(r \mu)\cos\left(\frac \pi 4 - y \left(\sinh \mu  - \mu \cosh \mu \right)\right)\right] \left(1 +o(1) \right)
\end{align*} as $y \rightarrow \infty$.
\end{theo}

\noindent For a comparison of Theorems~\ref{thmFirstCaseMonoKBessel} and~\ref{thmSecondCaseOscillKBessel} with related results in the literature, see~\cite[Remarks~1.2 and~1.4]{Ts18}.  In~\cite[Remark~1.4]{Ts18}, note that $=$ should replace $\sim$ and that the right-hand side of the displayed expression should be multiplied by $(1 + o(1))$.  Also, there are some related uniform results~\cite{Bal66, Ts18}.  See~\cite[Section~1 and Remark~1.7]{Ts18} for the definition of {\em uniform} and a discussion.

Finally, as part of our proof of Theorem~\ref{thmSecondCaseOscillKBessel}, we will show an integral representation of $K_\nu(y)$ that involves only one saddle point, namely the saddle point $R^+_0:=\mu+i \frac \pi 2$ (see Section~\ref{subsecSecondCaseYLessT} for the definition of saddle point and the values of the other saddle points).  The integrals of this integral representation are over the contour $w(u)$ given by (\ref{eqnPathSteepDesCase2}) and shown in Figure~\ref{fig:PathNovelIntRep}.  This integral representation is given in Theorem~\ref{propTemmOscIntRepKBessel2} and is a new result.

\begin{theo}  \label{propTemmOscIntRepKBessel2}  Let $\mu > 0$ be a fixed real number, $0< y \in \RR$, and $t  = y \cosh \mu$.
 Then we have the following integral representation:
 
\begin{align}\label{eqnTemmOscIntRepKBessel2} K_\nu(y) =&\frac 1 2 e^{i \chi}  \int_{\mu_-}^{\infty} e^{-y \psi(u)} e^{-ru}e^{irw}\left(1-i \frac{\wrt w}{\wrt u}\right) ~\wrt u \\\nonumber +&\frac 1 2 e^{i \chi}  \frac{e^{-2 \pi t + i 2 \pi r}}{1-e^{-2 \pi t + i 2 \pi r}}\int_{-\frac \pi 2}^{\frac 3 2 \pi } e^{-y \psi(u)} e^{-ru}e^{irw}\left(- \frac{\wrt u}{\wrt w}+i\right)~\wrt w
\\ \nonumber +&\frac 1 2 e^{-i\chi}  \int_{\mu_-}^{\infty} e^{-y \psi(u)} e^{ru}e^{irw}\left(1+i \frac{\wrt w}{\wrt u}\right) ~\wrt u 
\\ \nonumber -&\frac 1 2 e^{-i\chi}  \frac{e^{-2 \pi t + i 2 \pi r}}{1-e^{-2 \pi t + i 2 \pi r}}\int_{-\frac {\pi}2}^{\frac 3 2 \pi }e^{-y \psi(u)} e^{ru}e^{irw}\left(\frac{\wrt u}{\wrt w} +i\right)~\wrt w
\end{align} where $\psi(u) := \cosh u \cos w +w \cosh \mu$ and $\chi :=  \sqrt{t^2 - y^2}  - t \cosh^{-1}\left(\frac t y\right)$.
\end{theo}

%

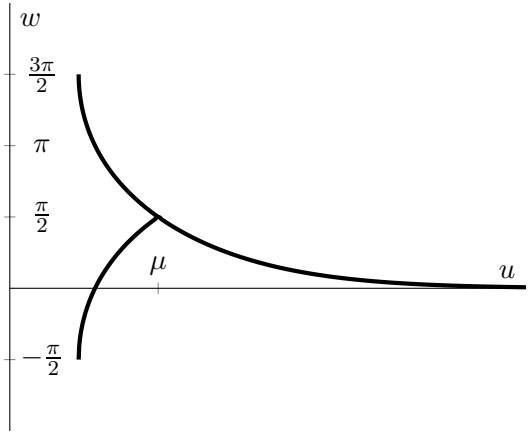
\begin{figure}[hbt!]
\begin{center}
\begin{tikzpicture}
\begin{axis}[graph default, xticklabels={\empty}, yticklabels={\empty}, xtick={2.302585092994046},  ytick={-1.570796326794897, 1.570796326794897, 3.141592653589793, 4.71238898038469}, xmin=0, xmax=8, ymin=-pi, ymax=2*pi]

      \addplot[ultra thick, domain=1.068075:10, samples=1000] {asin((101*x+99-101*ln(10))/(20*sinh (x)))/180*pi};
     
         \addplot[ultra thick, domain=1.068075:2.302585092994046, samples=1000] {pi - asin((101*x+99-101*ln(10))/(20*sinh (x)))/180*pi};

\node at (axis cs:0.5,-pi/2){$-\frac{\pi} 2$};
\node at (axis cs:0.5,pi/2){$\frac{\pi} 2$};
      \node at (axis cs:0.5,pi){$\pi$};
          \node at (axis cs:0.5,3*pi/2){$\frac{3\pi} 2$};
     
     \node at (axis cs:2.302585092994046,0.5){$\mu$};
    \end{axis}
\end{tikzpicture}
\end{center}
\caption{\label{fig:PathNovelIntRep} The contour over which the integrals in Theorem~\ref{propTemmOscIntRepKBessel2} are defined.  The only saddle point on the contour is $R^+_0= \mu + i\frac {\pi} 2$. Note that the two finite endpoints of the contour are $\mu_-- i \frac{\pi}2$ and $\mu_-+ i \frac{3\pi}2$.}  
\end{figure}

\subsection{Outline of paper}  Section~\ref{secBackBesselSaddlePoint} is devoted to defining the $K$-Bessel function and to giving a brief synopsis of the saddle-point method.  Section~\ref{subsecFirstCaseMonoKBessel} is devoted to the proof of the Theorem~\ref{thmFirstCaseMonoKBessel} and Section~\ref{subsecSecondCaseYLessT} is devoted to the proofs of Theorems~\ref{thmSecondCaseOscillKBessel} and~\ref{propTemmOscIntRepKBessel2}.  

\section{Background on the $K$-Bessel function and the saddle point method}\label{secBackBesselSaddlePoint}  The $K$-Bessel function is a solution to the differential equation (known as the {\em modified Bessel equation}) \[\frac{ \Wrt2 w}{\wrt z^2}+ \frac 1 z \frac{\wrt w}{\wrt z} - \left(1+\frac{\nu^2}{z^2}\right)w =0\] and can be defined as~\cite[Page~168]{Mac1899} (or see~\cite[Sections~10.25 and~10.27]{DLMF}) \begin{align}\label{eqnKInTermsOfI}
K_{\nu}(z) := \frac 1 2 \pi \frac{I_{-\nu}(z) - I_\nu(z)}{\sin(\nu \pi)}  \end{align} where $I_\nu(z)$ is the modified Bessel function of the first kind \[I_\nu(z) := \sum_{m=0}^\infty \frac{(\frac1 2 z)^{\nu+2 m}}{m! \Ga(\nu + m+1)}.\]  (The $\Ga(\cdot)$ here is the gamma function.)  Chapters~2 and~7 in~\cite{Olv} give a good introduction to the $K$-Bessel function, and a standard reference for Bessel functions, including the $K$-Bessel function, is~\cite{Wa}.

Of particular importance for us and for many other applications are integral representations of the $K$-Bessel function.  We will derive a new one in this paper, but we start with a well-known one (see~\cite[Page~182~(7)]{Wa} or~\cite[Equation~10.32.9]{DLMF} for example):  \begin{align}\label{IntRepKBessel}
K_{\nu}(z) = \frac 1 2 \int_{-\infty}^\infty e^{-z  \cosh R - \nu R } ~\wrt R = \frac 1 2 \int_{-\infty}^\infty e^{-z  \cosh R + \nu R } ~\wrt R
\end{align} where $z \in \CC \backslash \{0\}$ such that $|\arg(z)|< \frac{\pi}{2}$. 

To compute asymptotics for the $K$-Bessel function, it illustrative to apply to (\ref{IntRepKBessel}) the saddle-point method.   An elementary introduction to the saddle-point method is~\cite[Chapter~8]{Cop} (see also~\cite[Chapter~7]{Cop}).  Let us now give a rough explanation of how the saddle-point method works and leave the detailed proof of how it applies to the $K$-Bessel function to Section~\ref{subsecBndsBessel}.  Let $y \in \RR$ and $w \in \CC$.  Roughly speaking, the saddle-point method allows us to find an asymptotic, as $y \rightarrow \infty$, for an integral of the form \[\int_\gamma e^{-yp(w)}q(w)~\wrt w\]  where $\gamma$ is a contour and the functions $p(w)$ and $q(w)$ are both analytic.  Since the integrand is analytic, we have, by the Cauchy-Goursat theorem, considerable freedom in deforming $\gamma$ into a more suitable contour.  The path of steepest descent is a suitable contour (but not the only one) and it is characterized by \begin{itemize}
\item passing through a zero of $p'(w)$ (i.e. a {\em saddle point}),
\item having $\Im(p(w))$ constant along it, and
\item having the saddle point be a local maximum of $\Re(-p(w))$ along it.
\end{itemize}  Only the saddle points that are global maxima of $\Re(-p(w))$ (referred to as {\em of highest height}) along the path of steepest descent are needed for the asymptotic.  Often there are only a few of these.  For the $K$-Bessel function, there are either one or two of these saddle points depending on the case (and, in the case that there are two, we will reduce it to one by symmetry).  The saddle points of highest height are the points along the path of steepest descent where the integrand is largest when $y$ is large.  Moreover, the integral should be dominated by its restriction to suitable small neighborhoods around the saddle points of highest height.  The saddle-point method, then, is transparent and explanative:  it consists of finding these neighborhoods and verifying that the integral is indeed dominated by this restriction when $y$ is large.  Finally, we note that it is not necessary to only use paths of steepest descent, but there are restrictions on which contours are suitable for the method.

%
%
%
%

\section{Bounds for  $K_{r+it}(y)$ where $|t|$ large, $r$ bounded, and $y$ is real and positive}\label{subsecBndsBessel}



We will now apply the saddle-point method to finding an asymptotic for $K_\nu(y)$ as $y \rightarrow \infty$.  The saddle points and paths of steepest descent for the function $K_{it}(y)$ (i.e. purely imaginary order) have been obtained by N.~M.~Temme~\cite{Tem}.  The saddle points and paths of steepest descent for our function $K_{\nu}(y)$ are the same as we will prove below.  

We will start with the integral representation (\ref{IntRepKBessel}).  There are two cases:  $y \geq t \geq 0$ and $0 < y \leq t$.

\begin{rema}\label{rmkNegativeTIsOK}
Note that if $t <0$, then applying (\ref{IntRepKBessel}) allows us to be in one of these two cases.

\end{rema}

%
%
%
%
%
%
%

\subsection{First case:  $y \geq t \geq 0$} \label{subsecFirstCaseMonoKBessel}  In this section, we give the proof of Theorem~\ref{thmFirstCaseMonoKBessel}.

\begin{proof}[Proof of Theorem~\ref{thmFirstCaseMonoKBessel}]
 The cases $\theta =0$ (equivalently, $t=0$) and $\theta = \pi/2$ (or, equivalently, $t=y$) are exceptional and we leave them to the end.  Let us assume that $0<\theta < \pi/2$. Using (\ref{IntRepKBessel}), we have \begin{align}\label{IntRepKBessel2}
 K_\nu(y) = \frac 1 2 \int_{-\infty}^\infty e^{-y \varphi(R)} e^{rR} ~\wrt R  \end{align}where \[\varphi(R):= \cosh R - i R \sin \theta.\]
 
 The saddle points (values of $R$ for which $\varphi'(R)=0$) are as follows~\cite{Tem} (see also~\cite[Section~2.1]{BST}
): \[R_k := i\left( (-1)^k \theta + k\pi\right), \quad k \in \ZZ.\]

Let us now write $R=u+iw$ and thus we have \begin{align*}
\Re(-\varphi(R)) =& -\cosh u \cos w - w \sin \theta \\
  \Im(-\varphi(R)) =& -\sinh u \sin w + u \sin \theta \end{align*}

The path of steepest descent through the saddle point $R_0 = i \theta$ is given by $ \Im(-\varphi(R)) =  \Im(-\varphi(R_0))$ and is the following curve~\cite{Tem} (see Figure~\ref{fig:PathSteepDesCase1}): \begin{align}\label{eqnPathSteepDescentCase1}
 w = \arcsin\left(\sin \theta \frac{u}{\sinh u} \right), \quad -\infty < u < \infty. \end{align}  We remark that $w'(0)=0$ and that $w'(u)$ is bounded over all $-\infty < u <\infty$.

\begin{figure}[hbt!]
\begin{center}
\begin{tikzpicture}
\begin{axis}[graph default, xticklabels={\empty}, yticklabels={\empty}, xtick style={draw=none},  ytick={1.57079632679}, xmin=-10, xmax=10, ymin=-0.1, ymax=pi/2+0.4]
     \addplot[ultra thick, domain=-10:10, samples=100] {asin(sin(deg (pi/4)))/180*pi*x/sinh(x)};
     \node at (axis cs:0.5,pi/2){$\frac{\pi} 2$};
     \node at (axis cs:0.5,pi/4+0.07){$\theta$};
    \end{axis}
\end{tikzpicture}
\end{center}
\caption{\label{fig:PathSteepDesCase1} The path of steepest descent for the first case $y \geq t \geq 0$ when $\theta=\frac{\pi}4$.  The equation of this path is given by (\ref{eqnPathSteepDescentCase1}).}
\end{figure}
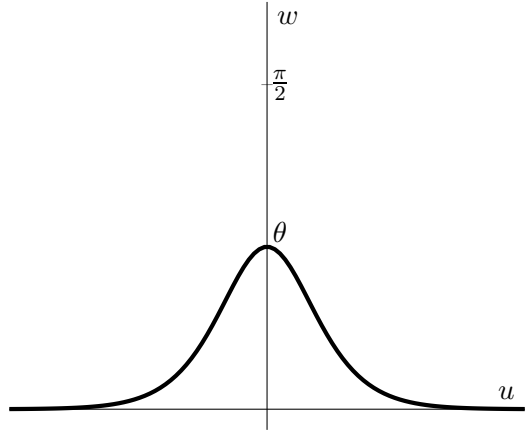

Using the Cauchy-Goursat theorem, we can replace the integral along the real axis from (\ref{IntRepKBessel}) with an integral along the path of steepest descent though the saddle point $R_0$.  Doing so, we obtain the following integral representation: \begin{align}\label{eqnAltIntRepKBesselMonotonic}
 K_\nu(y) = \frac 1 2 \int_{-\infty}^\infty e^{-y (\cosh u \cos w +w \sin \theta )} e^{ru}e^{irw}\left(1+i \frac{\wrt w}{\wrt u} \right)~\wrt u \end{align} whenever $y > t > 0$ holds.  (We will show below that this integral representation for $K_\nu(y)$ is also valid for the exceptional cases $t=0$ and $t=y$.) 
 
Consider the neighborhood of the saddle point $R_0$ along the path of steepest descent determined by $|u| < \min(\frac{\cot \theta} 2, \frac{\cot^4 \theta}{16},\frac{\theta} 4)$.  (Note that $0<\cot \theta< \infty$ as $0< \theta < \pi/2$.)   Using the Taylor series for $\sinh(x), 1/(1+x)$ and the geometric series, we have that \begin{align}\label{eqnBndsUSinhUKBessel} \frac{u}{\sinh u} = 1 - \kappa u^2+ O(u^4)
  \end{align} for $u$ in our small neighborhood.  Here $\kappa := \frac{1}{6}$ and the implied constant is less than $1/2$.  (See~\cite[Chapter 1, Section 2.2]{Olv} for the definition of the notion of {\em implied constant}.)  Thus,\begin{align}\label{eqnBndsCosSquaredWKBessel}
  \cos^2 w = 1 - \sin^2 \theta \left(\frac{u}{\sinh u}\right)^2 = \cos^2 \theta + 2\kappa u^2 \sin^2 \theta  + O(u^4) \end{align} for our small neighborhood.  Here the implied constant less than $\frac{10}{9} \sin^2 \theta$.  Using the Taylor series for $\sqrt{1 +x}$, we have that \begin{align}\label{eqnEstForcoswKBessel}
 \cos w = \cos \theta +\kappa  u^2 \sin \theta \tan \theta + O(u^3) \end{align} for $u$ in our small neighborhood.  Here the implied constant is less than $1/2$.  Note that, since that $0<w \leq \theta$, $\cos w \geq0$.
  

By the mean value theorem, we have that there exists $\widetilde{w} \in (w, \theta)$ such that \[\theta - w = \frac{\sin \theta - \sin w}{\cos\widetilde{w}},\] which yields \begin{align}\label{eqnEstForwKBessel}
 w = \theta - \frac{\kappa \sin \theta}{\cos \widetilde{w}}u^2+O(u^4) \end{align}  for our small neighborhood.  Note that $\cos \theta \leq \cos\widetilde{w} <1$.  Here the implied constant is less than $\frac{\sin \theta}{2 \cos \widetilde{w}} \leq \frac 1 2 \tan \theta$.
 
Pick $1/3 < \delta < 1/2$ and let \[y_0^{-\delta} = \begin{cases}  \min(\frac{\cot \theta} 2, \frac{\cot^4 \theta}{16},\frac{1}{2|r|}, \frac{\theta} 4) & \textrm{ if } |r|>1 \\ \min(\frac{\cot \theta} 2, \frac{\cot^4 \theta}{16}, \frac{\theta} 4) & \textrm{ if } |r| \leq 1 \end{cases}.\]  We now estimate the integral in (\ref{eqnAltIntRepKBesselMonotonic}) around a small neighborhood of the saddle point $R_0$ along the path of steepest descent determined by $|u| < y^{-\delta}$ where $y \geq y_0$.   As we are interested in the behavior of $K_\nu(y)$ as $y \rightarrow \infty$, we may assume that $y$ is large.  Using the proceeding estimates and the Taylor series for $\cosh(x)$, we have \begin{align*}
e^{-y (\cosh u \cos w +w \sin \theta )} &= e^{-y (\cos \theta +\theta \sin \theta )}e^{-y A u^2} e^{-y B u^3} \\&=  e^{-y (\cos \theta +\theta \sin \theta )}e^{-y A u^2} (1 + O(y^{1-3\delta}))  \end{align*} and \begin{align*} &e^{ru} = 1 +O(y^{-\delta}), \quad e^{irw-ir\theta} = 1 + O(y^{-2\delta}), \\ &\left(1+i \frac{\wrt w}{\wrt u} \right)= 1 -2i \frac{\kappa \sin \theta}{\cos \widetilde{w}}u+O(u^3) = 1 +O(y^{-\delta})
  \end{align*} over our small neighborhood.   Here $A = \frac 1 2 \cos \theta + \kappa \sin \theta  \tan \theta- \kappa \frac{\sin^2 \theta}{\cos \widetilde{w}}>0$ and $|B| < 11/12$.  Note that since we are on a small neighborhood around $0$, the approximation $e^x = 1 +O(x)$ holds.   Thus, we have 
\begin{align*} \frac 1 2 \int_{-y^{-\delta}}^{y^{-\delta}} &e^{-y (\cosh u \cos w +w \sin \theta )} e^{ru}e^{irw} \left(1+i \frac{\wrt w}{\wrt u} \right)~\wrt u \\ &= \frac 1 2 e^{-y (\cos \theta +\theta \sin \theta )}e^{ir\theta}\left(1 +O(y^{1-3\delta})\right)\int_{-y^{-\delta}}^{y^{-\delta}} e^{-y A u^2}~\wrt u. \end{align*}   (Note that $3 \delta -1 < \delta$ for our choice of $\delta$.)

Let $A_0 = \frac1 2 \cos \theta$. Now note that $A$ is a function of $y$ where \[
\lim_{y \rightarrow \infty} A = A_0 \textrm{ and } 0<A \leq  A_0.\]  Pick an $\varepsilon>0$ and set $\widetilde{A}:= \frac{A_0}{1 + \varepsilon}$.  Following the method in Copson~\cite[Chapter~8, (36.6) ff]{Cop}, namely changing variables $u^2 \mapsto y \widetilde{A} u^2$ and estimating, we obtain \[\int_{-y^{-\delta}}^{y^{-\delta}} e^{-y \widetilde{A} u^2}~\wrt u = \sqrt{\frac \pi {y \widetilde{A}}} \left(1 +o(y^{1-3 \delta}) \right).\]  Similarly, we have \[\int_{-y^{-\delta}}^{y^{-\delta}} e^{-y A_0 u^2}~\wrt u = \sqrt{\frac \pi {y A_0}} \left(1 +o(y^{1-3 \delta})\right).\]  (Note that it is important for $ \delta < 1/2$ here.)

Then, for every $\varepsilon>0$, there exists $y_1>0$ such that, whenever $y \geq y_1$, we have that \[1 \leq \frac{\int_{-y^{-\delta}}^{y^{-\delta}} e^{-y A u^2}~\wrt u }{\int_{-y^{-\delta}}^{y^{-\delta}} e^{-y A_0 u^2}~\wrt u } \leq \frac{\int_{-y^{-\delta}}^{y^{-\delta}} e^{-y \widetilde{A} u^2}~\wrt u }{\int_{-y^{-\delta}}^{y^{-\delta}} e^{-y A_0 u^2}~\wrt u } = \sqrt{1 + \varepsilon}.\]  Taking the limit as $y \rightarrow \infty$ and noting that $\varepsilon$ is arbitrary, we have that\begin{align}\label{eqnDomBehKBessel}
  \frac 1 2 \int_{-y^{-\delta}}^{y^{-\delta}} e^{-y (\cosh u \cos w +w \sin \theta )} &e^{ru}e^{irw} \left(1+i \frac{\wrt w}{\wrt u} \right)~\wrt u \sim \\ \nonumber &\sqrt{\frac{\pi}{2 y \cos \theta}}e^{-y (\cos \theta +\theta \sin \theta )}e^{ir\theta}\end{align} as $y \rightarrow \infty$.  (Compare with~\cite[(36.7)]{Cop}.)  Note the agreement with \cite[Page 87 (18)]{EMOT}~and~\cite[(14)]{BST} in the special case of purely imaginary order.
  
We now consider the rest of the integral, which we now show is negligible.  We will first integrate along the contour $\ell^+:=\{u + i w(y^{-\delta}): u\geq y^{-\delta}\}$ and then use the Cauchy-Goursat theorem.  The details are as follows.  Let \begin{align}\label{eqnBndYKBessel}
 y \geq \max\left(y_0, \left(\frac{4|r|}{\cos \theta}\right)^{\frac{1}{1-\delta}}\right). \end{align}  The integral along this contour is given by \[C:= \frac 1 2 \int_{y^{-\delta}}^\infty e^{-y \varphi\left(u +i w(y^{-\delta})\right)} e^{r\left(u +i w(y^{-\delta})\right) }~\wrt u\]  Note that the integrand comes from (\ref{IntRepKBessel2}).  Now we have that \begin{align*}
|C| \leq & \frac 1 2 \int_{y^{-\delta}}^\infty e^{-y\left(\cosh u \cos w(y^{-\delta}) + w(y^{-\delta}) \sin \theta\right)}e^{|r|u}~\wrt u  \\ \leq & \frac 1 2 \int_{y^{-\delta}}^\infty e^{-y\left((1 + \frac{u^2}{2}) \cos w(y^{-\delta}) +w(y^{-\delta}) \sin \theta\right)}e^{|r|u}~\wrt u \\\ \leq & \frac 1 2 \int_{y^{-\delta}}^\infty e^{-y\left( \frac{u^2}{2}\cos w(y^{-\delta})+\left(\cos \theta +\kappa  y^{-2\delta} \sin \theta \tan \theta + O(y^{-3\delta})\right) + \left(\theta - \frac{\kappa \sin \theta}{\cos \widetilde{w}}y^{-2\delta}+O(y^{-4\delta})\right) \sin \theta\right)}e^{|r|u}~\wrt u \end{align*}  Here we have used the fact that $\cosh u \geq 1 + u^2/2$ and applied (\ref{eqnEstForcoswKBessel}, \ref{eqnEstForwKBessel}).  Since $\kappa \sin \theta \tan \theta - \kappa \sin^2 \theta/\cos \widetilde{w}\geq 0$ and $\cos w \geq \cos \theta$, we have that \begin{align*}
|C| \leq& \frac {e^{-y(\cos \theta + \theta \sin \theta)}} {2} \int_{y^{-\delta}}^\infty e^{-y\left( \frac{u^2}{2}\cos \theta\right)}e^{|r|u}\left(1+ O(y^{(1-3 \delta)}\right)~\wrt u \\  \leq& \frac {e^{-y(\cos \theta + \theta \sin \theta)}} {2} \int_{y^{-\delta}}^\infty e^{-y\left( \frac{u^2}{4}\cos \theta\right)} \left(1+ O(y^{(1-3 \delta)}\right)~\wrt u \\  =& \frac {e^{-y(\cos \theta + \theta \sin \theta)}} {\sqrt{y \cos \theta}} \int_{\frac{y^{1/ 2-\delta} \sqrt{\cos \theta}}2}^\infty e^{-u^2 } \left(1+ O(y^{(1-3 \delta)}\right)~\wrt u \\  \leq& \frac {e^{-y(\cos \theta + \theta \sin \theta)}} {\sqrt{y \cos \theta}} \frac{e^{-\frac {y^{1-2 \delta}} 4  \cos \theta}}{\frac{y^{1/ 2-\delta} \sqrt{\cos \theta}}2 + \sqrt{\frac {y^{1-2 \delta}} 4  \cos \theta+ \frac 4 \pi}}  \left(1+ O(y^{(1-3 \delta)}\right) \end{align*} where the second inequality follows from (\ref{eqnBndYKBessel}), the equality from changing variables $y u^2\cos \theta /4 \mapsto u^2$, and the final inequality from a standard bound for $\operatorname{erfc}(x)$ (see~\cite[7.1.13]{AS64} for example).  This shows that, for large $y$, $|C|$ is negligible compared to the dominant behavior that we computed in (\ref{eqnDomBehKBessel}).

Now let us integrate over the contour $\ell^-(U):=\{U+i W:  0\leq W \leq w(y^{-\delta})\}$ for some $U \geq y^{-\delta}$: \[D :=  \frac 1 2 \int_0^{w(y^{-\delta})} e^{-y \varphi\left(U+i W)\right)} e^{r\left(U+i W)\right) }~\wrt W.\]  When $U$ is large enough, we have that \[|D| \leq e^{-y \cosh U \cos \theta} e^{|rU|} \theta,\] which for $U$ large enough is negligible compared to (\ref{eqnDomBehKBessel}).  Let $s^+(U) := \{v + i w(y^{-\delta}): U \geq v\geq y^{-\delta}\}$ and $s^-(U) := \{U+i W:  w(W) \leq W \leq w(y^{-\delta})\}$.  Then $s^+(U) \subset \ell^+$ and $s^-(U)\subset \ell^-(U)$ and $s^+(U) \cup s^-(U)$ meets a piece $c(U)$ of the path of steepest descent.  Then $s^+(U) \cup s^-(U) \cup c(U)$ is a simple closed contour and the Cauchy-Goursat theorem implies that the integral over $c(U)$ is negligible compared to (\ref{eqnDomBehKBessel}).  Letting $U \rightarrow \infty$ shows that the integral over the piece of the path of steepest descent for which $u \geq y^{-\delta}$ is also negligible compared to (\ref{eqnDomBehKBessel}).  For the integral over the remaining piece of the path of steepest descent, we note that $w(u), \cosh u$ are even functions and that the analogous proof also shows that it is negligible compared to (\ref{eqnDomBehKBessel}).  This proves the desired result in the case $0<\theta < \pi/2$.

We now prove the case $\theta = 0$.  The proof is a simplification of the previous case.  The details are as follows.  The integral representation for this case is \begin{align}\label{eqnAltIntRepKBesselMonotonic2}
 K_\nu(y) = \frac 1 2 \int_{-\infty}^\infty e^{-y \cosh u} e^{ru}~\wrt u \end{align} because $t=0$.  Pick $1/3 < \delta < 1/2$ and let $y_0:= 2^{1/\delta}$.  As in the previous case, the dominant behavior of the integral comes from a small neighborhood around the origin, namely $|u| < y^{-\delta}$ where $y \geq y_0$.  On this neighborhood, $\cosh u = 1 + u^2/2 + O(u^4)$.  Let us consider this contribution first:  \begin{align*}K_\nu(y) = \frac 1 2 \int_{-y^{-\delta}}^{y^{-\delta}} e^{-y \cosh u} e^{ru}~\wrt u =& \frac 1 2 e^{-y} \int_{-y^{-\delta}}^{y^{-\delta}} e^{-y \frac{u^2} 2} \left(1 + O(y^{1-4 \delta})\right) \left(1 + O(y^{-\delta})\right)~\wrt u \\=& \frac 1 2 e^{-y} \int_{-y^{-\delta}}^{y^{-\delta}} e^{-y \frac{u^2} 2} \left(1 + O(y^{-1/3})\right)~\wrt u \\=& \sqrt{\frac {\pi}{2y}}e^{-y} \left(1 + O(y^{-1/3})\right).
  \end{align*}  The final equality is obtained, as in the case $0<\theta < \pi/2$, by changing variables $u^2 \mapsto \frac y 2 u^2$ and estimating $\operatorname{erf}(x) = 1 - \operatorname{erfc}(x)$ with the standard bound for $\operatorname{erfc}(x)$.
  
 Away from this neighborhood, the integral is negligible.   Let \begin{align*}
 y \geq \max\left(y_0, \left(4|r|\right)^{\frac{1}{1-\delta}}\right). \end{align*}  We have that \begin{align*}\left |  \frac 1 2 \int_{y^{-\delta}}^\infty e^{-y \cosh u} e^{ru}~\wrt u\right| \leq  \frac 1 2 e^{-y}  \int_{y^{-\delta}}^\infty e^{-\frac y 4 u^2}~\wrt u  \leq \frac {e^{-y}} {\sqrt{y}} \frac{e^{-\frac {y^{1-2 \delta}} 4 }}{\frac{y^{1/ 2-\delta} }2 + \sqrt{\frac {y^{1-2 \delta}} 4 + \frac 4 \pi}}.\end{align*}  For the rest of the integral over $-\infty$ to $-y^{-\delta}$, we obtain a similar bound.  This gives the desired result for the case $\theta =0$.
 
 We now prove the final case of $\theta=\pi/2$.  Here $t=y$.  Unlike in the previous two cases, here the saddle points $R_k$ are no longer of order 1 but are, instead, of order 2 (see~\cite[Page~40]{Erd56} for the definition of order).\footnote{Pairs of saddle points when $\theta<\pi/2$ have coalesced at $\theta =\pi/2$.}  It still suffices to consider the saddle point $R_0 = i \pi/2$.  We now have that  \[w = \arcsin\left(\frac{u}{\sinh u} \right), \quad -\infty < u < \infty.\]  We remark that $w'(u)$ is a bounded, odd function.  It has a jump discontinuity at $u=0$.

Pick $1/4 < \delta < 1/3$.  Consider the neighborhood of the saddle point $R_0$ along the path of steepest descent determined by $|u| < 1/4$.  We have \begin{align}\label{eqnBndsUSinhU2KBessel} \frac{u}{\sinh u} = 1 - \kappa u^2+ \lambda u^4+O(u^6)
  \end{align} for $u$ in our small neighborhood.  Here $\kappa=\frac 1 6$, $\lambda := \frac 7 {320}$, and the implied constant is less than $1/2$.  Using this and the Taylor series for $\sqrt{1+x}$, we obtain \[\cos w = \sqrt{2 \kappa} |u| - \widetilde{\lambda} |u|^3 +O(u^5)\] where $\widetilde{\lambda}:= \frac{21}{200}$ and the implied constant is less than $2$.  (Note that the precise value of $\widetilde{\lambda}$ does not come into the computation of the dominant behavior.)  Now using the Taylor series for $\arccos(x)$, we have \[w = \frac \pi 2 - \sqrt{2 \kappa} |u| + \widetilde{\lambda} |u|^3  - \frac 1 6 \left( \sqrt{2 \kappa} |u| \right)^3+O(u^5),\]  and taking the derivative with respect to $u$ yields \[\frac{\wrt w}{\wrt u}= \begin{cases} - \sqrt{2 \kappa} +O(u^2) & \text{ if } 0^+ \leq u < 1/4  \\ \sqrt{2 \kappa} + O(u^2) & \text{ if } -1/4 < u \leq  0^-\end{cases}.\]  Using the Taylor series for $\cosh(x)$, we have \begin{align}\label{eqnCoshCoswEstpi2KBessel}
 \cosh u \cos w + w = \frac \pi 2 + \frac 4{9 \sqrt{3}}|u|^3+O(u^5). \end{align}  First note that \begin{align*} &\int_{-y^{-\delta}}^{y^{-\delta}} e^{-y \left( \frac 4{9 \sqrt{3}}|u|^3\right)} \left(1+i \frac{\wrt w}{\wrt u} \right)~\wrt u \\ =& \int_{0}^{y^{-\delta}} e^{-y \left( \frac 4{9 \sqrt{3}}u^3\right)} \left(1-i\sqrt{2 \kappa}+O(y^{-2\delta})  \right)~\wrt u  +\int_{-y^{-\delta}}^{0} e^{y \left( \frac 4{9 \sqrt{3}}u^3\right)} \left(1 +i\sqrt{2 \kappa}+O(y^{-2\delta}) \right)~\wrt u \\ =&  \int_{0}^{y^{-\delta}} e^{-y \left( \frac 4{9 \sqrt{3}}u^3\right)} \left(1-i\sqrt{2 \kappa}+O(y^{-2\delta}) + 1 +i\sqrt{2 \kappa}+O(y^{-2\delta})  \right)~\wrt u \\ =& (2 + O(y^{-1/4})) \int_{0}^{y^{-\delta}} e^{-y \left( \frac 4{9 \sqrt{3}}u^3\right)}~\wrt u  
  \end{align*} where the third equality follows from changing variables $u \mapsto -u$.
  

  Consequently, we have \begin{align*}
 \frac 1 2 &\int_{-y^{-\delta}}^{y^{-\delta}} e^{-y (\cosh u \cos w +w  )} e^{ru}e^{irw}\left(1+i \frac{\wrt w}{\wrt u} \right)~\wrt u \\ =& \frac 1 2 e^{-\frac \pi 2 y +i\frac \pi 2 r}\int_{-y^{-\delta}}^{y^{-\delta}} e^{-y \left( \frac 4{9 \sqrt{3}}|u|^3+O(u^5)\right)} e^{ru}e^{O(u)}\left(1+i \frac{\wrt w}{\wrt u} \right)~\wrt u \\ =& \frac 1 2 e^{-\frac \pi 2 y +i\frac \pi 2 r}\int_{-y^{-\delta}}^{y^{-\delta}} e^{-y \left( \frac 4{9 \sqrt{3}}|u|^3\right)} \left(1 +O(y^{1-5\delta})\right) \left(1 +O(y^{-\delta})\right)\left(1+i \frac{\wrt w}{\wrt u} \right)~\wrt u  \\ =& \frac 1 2 e^{-\frac \pi 2 y +i\frac \pi 2 r}\left(1 +O\left(y^{-1/4}\right)\right) \int_{-y^{-\delta}}^{y^{-\delta}} e^{-y \left( \frac 4{9 \sqrt{3}}|u|^3\right)} \left(1+i \frac{\wrt w}{\wrt u} \right)~\wrt u \\ =&\frac 1 2 e^{-\frac \pi 2 y +i\frac \pi 2 r} \left(2 + O(y^{-1/4})\right) \int_{0}^{y^{-\delta}} e^{-y \left( \frac 4{9 \sqrt{3}}u^3\right)}~\wrt u \\ =&\frac 1 2 e^{-\frac \pi 2 y +i\frac \pi 2 r}y^{-1/3}\left( \frac{9\sqrt{3}}{4}\right)^{1/3} \left(2 + O(y^{-1/4})\right) \int_{0}^{\left( 4/9\sqrt{3}\right)^{1/3} y^{1/3 -\delta}} e^{-u^3}~\wrt u \\ =&\frac 1 2 e^{-\frac \pi 2 y +i\frac \pi 2 r}y^{-1/3}\left( \frac{9\sqrt{3}}{4}\right)^{1/3} \left(2 + O(y^{-1/4})\right)\frac{\Gamma(\frac 1 3)-\Gamma(\frac 1 3,  \frac{4}{9\sqrt{3}} y^{1-3\delta})}{3} \\ =&\frac 1 2 e^{-\frac \pi 2 y +i\frac \pi 2 r}y^{-1/3}\left( \frac{9\sqrt{3}}{4}\right)^{1/3} \left(2 + O(y^{-1/4})\right)\left(\frac{\Gamma(\frac 1 3)} 3 +o(y^{-1/4})\right)\end{align*} where in the second-to-last equality we have changed variables $u \mapsto u^{1/3}$ and in last equality we have used a standard estimate for the incomplete gamma function $\Gamma(a, x)$ (see~\cite[6.5.32]{AS64} for example).  Consequently we have that \begin{align*}  \frac 1 2 \int_{-y^{-\delta}}^{y^{-\delta}} e^{-y (\cosh u \cos w +w  )} e^{ru}e^{irw}\left(1+i \frac{\wrt w}{\wrt u} \right)~\wrt u \sim e^{-\frac \pi 2 y +i\frac \pi 2 r}y^{-1/3}\frac{\Gamma(\frac 1 3)} {2^{\frac 2 3}3^{\frac 1 6}},
  \end{align*} which agrees with~\cite[Pages~78,~247]{Wa}~and~\cite[(14)]{BST} in the special case of purely imaginary order.

We now show that the rest of the integral is negligible.  As in the case $0< \theta < \pi/2$, we will use the Cauchy-Goursat theorem and the contours given by $C$ and $D$.  Let \begin{align}\label{eqnBndypi2CaseKBessel} y > \max\left(4^{\frac 1{\delta}}, \left(8\sqrt{3} |r| \right)^{\frac{1}{1-2\delta}}\right). \end{align} 
 We have that \begin{align*}
|C| \leq & \frac 1 2 \int_{y^{-\delta}}^\infty e^{-y\left(\cosh u \cos w(y^{-\delta}) + w(y^{-\delta}) \right)}e^{|r|u}~\wrt u  \\ \leq & \frac 1 2 \int_{y^{-\delta}}^\infty e^{-y\left((1 + \frac{u^2}{2}) \cos w(y^{-\delta}) +w(y^{-\delta}) \right)}e^{|r|u}~\wrt u \\\ \leq & \frac 1 2 \int_{y^{-\delta}}^\infty e^{-y\left( \frac{u^2}{2}\cos w(y^{-\delta})+\left(\sqrt{2 \kappa} y^{-\delta} - \widetilde{\lambda} y^{-3\delta} +O(y^{-5\delta})\right) + \left(\frac \pi 2 - \sqrt{2 \kappa} y^{-\delta} + \widetilde{\lambda} y^{-3\delta}  - \frac 1 6 \left( \sqrt{2 \kappa} y^{-\delta} \right)^3+O(y^{-5\delta})\right)\right)}e^{|r|u}~\wrt u \\ =& \frac 1 2e^{-\frac \pi 2 y + \frac {\left( \sqrt{2 \kappa}  \right)^3}{6}y^{1-3\delta} }\left( 1 +O(y^{1-5\delta})\right)\int_{y^{-\delta}}^\infty e^{-y \frac{u^2}{2}\left(\sqrt{2 \kappa} y^{-\delta} - \widetilde{\lambda} y^{-3\delta} +O(y^{-5\delta})  \right)}e^{|r|u}~\wrt u.
\end{align*} 

Note that since $0< y^{-\delta}\leq \frac 1 4$, we have that $\sqrt{2 \kappa} y^{-\delta} - \widetilde{\lambda} y^{-3\delta} +O(y^{-5\delta})\geq \frac1{2 \sqrt{3}} y^{-\delta}$ because the implied constant has norm less than $2$.  Applying (\ref{eqnBndypi2CaseKBessel}), we have that  \begin{align*}
|C| \leq & \frac 1 2e^{-\frac \pi 2 y + \frac {\left( \sqrt{2 \kappa}  \right)^3}{6}y^{1-3\delta} }\left( 1 +O(y^{1-5\delta})\right)\int_{y^{-\delta}}^\infty e^{-y \frac{u^2}{4}\left(\sqrt{2 \kappa} y^{-\delta} - \widetilde{\lambda} y^{-3\delta} +O(y^{-5\delta})  \right)}~\wrt u \\\leq & \frac 1 2e^{-\frac \pi 2 y + \frac {\left( \sqrt{2 \kappa}  \right)^3}{6}y^{1-3\delta} }\left( 1 +O(y^{1-5\delta})\right)\int_{y^{-\delta}}^\infty e^{-y^{1-\delta} \frac{u^2}{8 \sqrt{3}}}~\wrt u \\\leq & \frac {\sqrt{8 \sqrt{3}}}{2 y^{\frac{1-\delta} 2}}e^{-\frac \pi 2 y + \frac {\left( \sqrt{2 \kappa}  \right)^3}{6}y^{1-3\delta} }\left( 1 +O(y^{1-5\delta})\right)\int_{(8 \sqrt{3})^{-1/2}y^{\frac{1-3 \delta}2}}^\infty e^{-u^2}~\wrt u 
\\\leq & \frac {\sqrt{8 \sqrt{3}}}{2 y^{\frac{1-\delta} 2}}e^{-\frac \pi 2 y + \frac {\left( \sqrt{2 \kappa}  \right)^3}{6}y^{1-3\delta} }\left( 1 +O(y^{1-5\delta})\right)\frac{e^{-\frac {y^{1-3 \delta}}{8 \sqrt{3}}}}{\frac{y^{\frac{1-3 \delta}2}}{\sqrt{8 \sqrt{3}}}+\sqrt{\frac {y^{1-3 \delta}}{8 \sqrt{3}} + \frac 4 \pi}}.\end{align*}

Since $\frac{\left( \sqrt{2 \kappa}  \right)^3}{6} - \frac1{8 \sqrt{3}} = - \frac 5 {72 \sqrt{3}}$ and $\frac{1-\delta}2 > \frac 1 3$, we have that $|C|$ is negligible compared to the dominant behavior that we computed.

For $D$, we have the following estimate:  \begin{align*}|D| \leq & \frac 1 2 \int_{0}^{w(y^{-\delta})} e^{-y\left(\cosh U \cos W + W \right)}e^{|rU|}~\wrt W  \\ \leq & \frac 1 2 w(y^{-\delta})e^{-y\cosh U \cos (w(y^{-\delta}))}e^{|rU|}  \\ = & \frac 1 2 w(y^{-\delta})e^{-y\cosh U \left( \sqrt{2 \kappa} y^{-\delta} - \widetilde{\lambda} y^{-3\delta} +O(y^{-5\delta})\right)}e^{|rU|}   \\ = & \frac 1 2 w(y^{-\delta})e^{-\frac{y^{1-\delta}}{2 \sqrt{3}}\cosh U}e^{|rU|},
  \end{align*} which, when $U$ is large enough, is negligible compared to the dominant behavior that we computed.  Note that $\cos W \geq \cos (w(y^{-\delta}))$.  The integral over the remaining piece of the path of steepest descent is handled in a manner analogous to the $0< \theta < \pi/2$ case.

  This proves the desired result in all cases.

\end{proof}

\subsection{Second case:  $0 < y < t$}\label{subsecSecondCaseYLessT}  In this section, we prove Theorem~\ref{thmSecondCaseOscillKBessel} and, to do so, we must first prove Theorem~\ref{propTemmOscIntRepKBessel2}.  Let define the constant $\mu>0$ by $t = y \cosh \mu$ and the function \[\psi(u) := \cosh u \cos w +w \cosh \mu.\]  We start by finding the saddle points and suitable integral representations.

Using (\ref{IntRepKBessel}), we have \begin{align}\label{IntRepKBessel3}
 K_\nu(y) = \frac 1 2 \int_{-\infty}^\infty e^{-y \phi(R)} e^{rR} ~\wrt R  \end{align}where \[\phi(R):= \cosh R - i R\cosh \mu.\]
 
 The saddle points (values of $R$ for which $\phi'(R)=0$) are as follows~\cite{Tem} (see also~\cite[Section~2.1]{BST}
): \[R^\pm_k := \pm \mu + i\left( \frac \pi 2 + 2k\pi\right), \quad k \in \ZZ.\]

Let us now write $R=u+iw$ and thus we have \begin{align*}
\Re(-\phi(R)) =& -\cosh u \cos w  -w \cosh \mu = - \psi(u),\\
  \Im(-\phi(R)) =& -\sinh u \sin w + u \cosh \mu. \end{align*}

The paths of steepest descent/ascent through the saddle points $R^\pm_k$ is given by $ \Im(-\phi(R)) =  \Im(-\phi(R^\pm_k))$ and is the following family of curves~\cite{Tem}: \begin{align}\label{eqnPathSteepDesCase2}
  \sin w = \cosh \mu \frac {u}{\sinh u} \pm \frac{\sinh \mu - \mu \cosh \mu}{\sinh u}. \end{align}  We use only the parts of these curves as shown as solid line in~\ref{fig:PathSteepDesCase2}, which we will refer to as the path of steepest descent.  
  
  \begin{figure}[hbt!]
\begin{center}
\begin{tikzpicture}
\begin{axis}[graph default, xticklabels={\empty}, yticklabels={\empty}, xtick={-2.302585092994046,2.302585092994046},  ytick={1.570796326794897, 3.141592653589793, 4.71238898038469, 6.283185307179586, 7.853981633974483, 9.42477796076938, 10.995574287564276, 12.566370614359173, 14.13716694115407,  15.707963267948966, 17.278759594743863, 18.849555921538759, 20.420352248333656}, xmin=-8, xmax=8, ymin=-1.5, ymax=7*pi]
     \addplot[ultra thick, domain=2.302585092994046:10, samples=200] {asin((101*x+99-101*ln(10))/(20*sinh (x)))/180*pi};
     \addplot[ultra thick, domain=1.068075:2.302585092994046, samples=200] {pi - asin((101*x+99-101*ln(10))/(20*sinh (x)))/180*pi};
     \addplot[ultra thick, dotted, domain=2.302585092994046:10, samples=200] {pi - asin((101*x+99-101*ln(10))/(20*sinh (x)))/180*pi};
     \addplot[ultra thick, domain=1.068075:2.302585092994046, samples=200] {2*pi + asin((101*x+99-101*ln(10))/(20*sinh (x)))/180*pi};
     \addplot[ultra thick, dotted, domain=2.302585092994046:10, samples=200] {2*pi + asin((101*x+99-101*ln(10))/(20*sinh (x)))/180*pi};
      \addplot[ultra thick, domain=1.068075:2.302585092994046, samples=200] {3*pi - asin((101*x+99-101*ln(10))/(20*sinh (x)))/180*pi};
      \addplot[ultra thick, dotted, domain=2.302585092994046:10, samples=200] {3*pi - asin((101*x+99-101*ln(10))/(20*sinh (x)))/180*pi};
      \addplot[ultra thick, domain=1.068075:2.302585092994046, samples=200] {4*pi + asin((101*x+99-101*ln(10))/(20*sinh (x)))/180*pi};
      \addplot[ultra thick, dotted, domain=2.302585092994046:10, samples=200] {4*pi + asin((101*x+99-101*ln(10))/(20*sinh (x)))/180*pi};
      \addplot[ultra thick, domain=1.068075:2.302585092994046, samples=200] {5*pi - asin((101*x+99-101*ln(10))/(20*sinh (x)))/180*pi};
      \addplot[ultra thick, dotted, domain=2.302585092994046:10, samples=200] {5*pi - asin((101*x+99-101*ln(10))/(20*sinh (x)))/180*pi};
      \addplot[ultra thick, domain=1.068075:2.302585092994046, samples=200] {6*pi + asin((101*x+99-101*ln(10))/(20*sinh (x)))/180*pi};
      \addplot[ultra thick, dotted, domain=2.302585092994046:10, samples=200] {6*pi + asin((101*x+99-101*ln(10))/(20*sinh (x)))/180*pi};
       \addplot[ultra thick, domain=1.068075:2.302585092994046, samples=200] {7*pi - asin((101*x+99-101*ln(10))/(20*sinh (x)))/180*pi};
       \addplot[ultra thick, dotted, domain=2.302585092994046:10, samples=200] {7*pi - asin((101*x+99-101*ln(10))/(20*sinh (x)))/180*pi};

       \addplot[ultra thick, domain=-10:-2.302585092994046, samples=200] {asin((101*x-99+101*ln(10))/(20*sinh (x)))/180*pi};
     \addplot[ultra thick, domain=-2.302585092994046:-1.068075, samples=200] {pi - asin((101*x-99+101*ln(10))/(20*sinh (x)))/180*pi};
     \addplot[ultra thick, dotted, domain=-10:-2.302585092994046, samples=200] {pi - asin((101*x-99+101*ln(10))/(20*sinh (x)))/180*pi};
     \addplot[ultra thick, domain=-2.302585092994046:-1.068075, samples=200] {2*pi + asin((101*x-99+101*ln(10))/(20*sinh (x)))/180*pi};
     \addplot[ultra thick, dotted, domain=-10:-2.302585092994046, samples=200] {2*pi + asin((101*x-99+101*ln(10))/(20*sinh (x)))/180*pi};
      \addplot[ultra thick, domain=-2.302585092994046:-1.068075, samples=200] {3*pi - asin((101*x-99+101*ln(10))/(20*sinh (x)))/180*pi};
      \addplot[ultra thick, dotted, domain=-10:-2.302585092994046, samples=200] {3*pi - asin((101*x-99+101*ln(10))/(20*sinh (x)))/180*pi};
      \addplot[ultra thick, domain=-2.302585092994046:-1.068075, samples=200] {4*pi + asin((101*x-99+101*ln(10))/(20*sinh (x)))/180*pi};
      \addplot[ultra thick, dotted, domain=-10:-2.302585092994046, samples=200] {4*pi + asin((101*x-99+101*ln(10))/(20*sinh (x)))/180*pi};
      \addplot[ultra thick, domain=-2.302585092994046:-1.068075, samples=200] {5*pi - asin((101*x-99+101*ln(10))/(20*sinh (x)))/180*pi};
      \addplot[ultra thick, dotted, domain=-10:-2.302585092994046, samples=200] {5*pi - asin((101*x-99+101*ln(10))/(20*sinh (x)))/180*pi};
      \addplot[ultra thick, domain=-2.302585092994046:-1.068075, samples=200] {6*pi + asin((101*x-99+101*ln(10))/(20*sinh (x)))/180*pi};
      \addplot[ultra thick, dotted, domain=-10:-2.302585092994046, samples=200] {6*pi + asin((101*x-99+101*ln(10))/(20*sinh (x)))/180*pi};
       \addplot[ultra thick, domain=-2.302585092994046:-1.068075, samples=200] {7*pi - asin((101*x-99+101*ln(10))/(20*sinh (x)))/180*pi};
       \addplot[ultra thick, dotted, domain=-10:-2.302585092994046, samples=200] {7*pi - asin((101*x-99+101*ln(10))/(20*sinh (x)))/180*pi};

      \node at (axis cs:0.7,pi){$\pi$};
     \node at (axis cs:0.7,2*pi){$2 \pi$};
     \node at (axis cs:0.7,3*pi){$3 \pi$};
     \node at (axis cs:0.7,4*pi){$4 \pi$};
     \node at (axis cs:0.7,5*pi){$5 \pi$};
     \node at (axis cs:0.7,6*pi){$6 \pi$};
     
     \node at (axis cs:-2.302585092994046,-0.8){$-\mu$};
     \node at (axis cs:2.302585092994046,-0.8){$\mu$};
    \end{axis}
\end{tikzpicture}
\end{center}
\caption{\label{fig:PathSteepDesCase2} The path of steepest descent for the second case $0 < y \leq t$ is represented by the solid line.  The equation of this path comes from (\ref{eqnPathSteepDesCase2}). The dotted lines are other curves that come from (\ref{eqnPathSteepDesCase2}) and are not part of the path of steepest descent.}  
\end{figure}
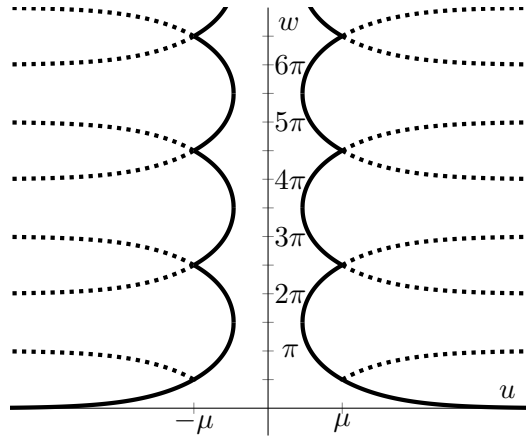

\noindent Notice that this path is the union of two branches $\eL^- \cup \eL^+$, separated by the imaginary axis, where \begin{align*}
 \textrm{ --- }& \eL^- \textrm { runs from } -\infty \textrm { to } 0 \textrm { and from } 0 \textrm{ to } + i\infty, \\  \textrm{ --- }& \eL^+ \textrm { runs from } +i\infty \textrm { to } 0 \textrm { and from } 0 \textrm{ to } + \infty. \end{align*}  What is important about this path is that, on both of the branches, the function $y\phi(R)$ has constant imaginary part, namely  \begin{align*} \chi := \Im(y\phi(R_0^+)):=&y \left(\sinh \mu  - \mu \cosh \mu \right) \\ =& y \sinh \mu - t \cosh^{-1}\left(\frac t y\right)= \sqrt{t^2 - y^2}  - t \cosh^{-1}\left(\frac t y\right), \\   \chi_- := \Im(y\phi(R_0^-)) =& - \chi
  \end{align*} for $\eL^+$ and $\eL^-$, respectively.  
  
Using the Cauchy-Goursat theorem, we can replace the integral along the real axis from (\ref{IntRepKBessel}) with an integral along the path of steepest descent:\begin{align}K_\nu(y) =& \frac 1 2 e^{-i\chi_-} \left(\int_{-\infty}^{-\mu} e^{-y \psi(u)} e^{ru}e^{irw}\left(1+i \frac{\wrt w}{\wrt u}\right) ~\wrt u ~ + \int_{\frac 1 2 \pi}^\infty  e^{-y \psi(u)} e^{ru}e^{irw}\left( \frac{\wrt u}{\wrt w}+i \right) ~\wrt w \right) \\\nonumber +&\frac 1 2 e^{-i\chi} \left(- \int_{\frac 1 2 \pi}^\infty  e^{-y \psi(u)} e^{ru}e^{irw}\left( \frac{\wrt u}{\wrt w}+i \right) ~\wrt w  +\int_{\mu}^{\infty} e^{-y \psi(u)} e^{ru}e^{irw}\left(1+i \frac{\wrt w}{\wrt u}\right) ~\wrt u \right).\end{align}  

Let us first show that we can obtain the integral representation:
\begin{prop}\label{propTemmOscIntRepKBessel1}  We have
 
\begin{align}\label{eqnTemmOscIntRepKBessel} K_\nu(y) =&\frac 1 2 e^{i \chi}  \int_{\mu}^{\infty} e^{-y \psi(u)} e^{-ru}e^{irw}\left(1-i \frac{\wrt w}{\wrt u}\right) ~\wrt u \\ \nonumber +&\frac 1 2 e^{-i\chi}  \int_{\mu}^{\infty} e^{-y \psi(u)} e^{ru}e^{irw}\left(1+i \frac{\wrt w}{\wrt u}\right) ~\wrt u \\\nonumber +&\frac 1 2 e^{i \chi}  \frac{1}{1-e^{-2 \pi t + i 2 \pi r}}\int_{\frac 1 2 \pi}^{\frac 5 2 \pi} e^{-y \psi(u)} e^{-ru}e^{irw}\left(- \frac{\wrt u}{\wrt w}+i\right)~\wrt w
\\ \nonumber -&\frac 1 2 e^{-i\chi}  \frac{1}{1-e^{-2 \pi t + i 2 \pi r}}\int_{\frac 1 2 \pi}^{\frac 5 2 \pi}e^{-y \psi(u)} e^{ru}e^{irw}\left(\frac{\wrt u}{\wrt w} +i\right)~\wrt w
\end{align} where the integrals are only over $\eL^+$.   
\end{prop}
\begin{rema}
In the special case of purely imaginary order, (\ref{eqnTemmOscIntRepKBessel}) reduces to~\cite[(3.5)]{Tem} and the proof of~(\ref{eqnTemmOscIntRepKBessel}) is similar to that in~\cite{Tem}.
\end{rema}
\begin{proof}

To begin the proof of (\ref{eqnTemmOscIntRepKBessel}), first note that \begin{align*} w(u) &= w(-u) \\ \psi(u) &= \psi(-u).
  \end{align*}  Also, $u(w)$ is two-valued, one for each branch.  Let us, for clarity, temporarily use $u^+(w)$ to denote the value on $\eL^+$ and $u^-(w)$ to denote the value on $\eL^-$, then we have that \[u^+(w) = - u^-(w) \quad \quad \frac{\wrt u^+}{\wrt w} = - \frac{\wrt u^-}{\wrt w}.\]

 Now let us change variables $u \mapsto -u$ on the branch $\eL^-$ so that all integrals will be over only the branch $\eL^+$:   \begin{align*}\int_{-\infty}^{-\mu} e^{-y \psi(u)} e^{ru}e^{irw}\left(1+i \frac{\wrt w}{\wrt u}\right) ~\wrt u =& \int_{\mu}^{\infty} e^{-y \psi(u)} e^{-ru}e^{irw}\left(1-i \frac{\wrt w}{\wrt u}\right) ~\wrt u \\
 \int_{\frac 1 2 \pi}^\infty  e^{-y \psi(u)} e^{ru}e^{irw}\left( \frac{\wrt u}{\wrt w}+i \right) ~\wrt w =&\int_{\frac 1 2 \pi}^\infty  e^{-y \psi(u)} e^{-ru}e^{irw}\left(- \frac{\wrt u}{\wrt w }+i  \right) ~\wrt w.
  \end{align*}
  
We now have two integrals with respect to $\wrt w$, and these are both integrated over a piece of $\eL^+$.  Moreover, these integrals can both be reduced to integrals over a finite interval.  Given \begin{align*}
P(w) := e^{-y \cosh u \cos w+ru}\left( \frac{\wrt u}{\wrt w}+i \right)  \quad \quad
Q(w) := e^{-y \cosh u \cos w-ru}\left(- \frac{\wrt u}{\wrt w}  +i\right),
\end{align*} we have that $P(w) = P(w+2\pi)$ and $Q(w) = Q(w+2\pi)$ because $u(w) = u(w+2\pi)$ holds over the bounds of integration of the integrals with respect to $\wrt w$.  Consequently, using the geometric series (which is valid because $t>0$), we have that \begin{align*}\int_{\frac 1 2 \pi}^\infty P(w) e^{-tw+irw}~\wrt w =& \frac{1}{1-e^{-2 \pi t + i 2 \pi r}}\int_{\frac 1 2 \pi}^{\frac 5 2 \pi} P(w) e^{-tw+irw}~\wrt w \\ \int_{\frac 1 2 \pi}^\infty Q(w) e^{-tw+irw}~\wrt w =& \frac{1}{1-e^{-2 \pi t + i 2 \pi r}}\int_{\frac 1 2 \pi}^{\frac 5 2 \pi} Q(w) e^{-tw+irw}~\wrt w
  \end{align*}  Combining all of this gives (\ref{eqnTemmOscIntRepKBessel}), as desired.

\end{proof}

Now take the piece of $\eL^+$ from $w = \frac 3 2 \pi$ to $w = \frac 5 2 \pi$ (inclusive of the endpoints) and shift it down the vertical axis by $2 \pi$.  Call this shifted piece $\eL^0$.  Note that $\eL^0$ does not meet $\eL^-$ and meets $\eL^+$ only at one point (namely, the saddle point $R_0^+$).  Although $w(u)$ is multiple-valued on $\eL^+ \cup \eL^0$, we note that $w(u)$ is single-valued on $\eL^0$ and on the piece of $\eL^+$ corresponding to $0\leq w \leq3\pi/2$.

Let $\mu_-:=u(-\pi/2) = u(3\pi/2)$.  For the task of computing the dominant behavior, we need to modify the integral representation~(\ref{eqnTemmOscIntRepKBessel}) in two ways.  The first is to shift the integrals from $\frac 3 2 \pi$ to $\frac 5 2 \pi$ by $-2\pi$ (and this piece of the integrals will be over $\eL^0$) and the second is to convert some of the integrals with respect to $w$ to integrals with respect to $u$.  Now we can derive the integral representation that we need for the proof of the second case, namely Theorem~\ref{propTemmOscIntRepKBessel2}.  See Figure~\ref{fig:PathNovelIntRep} for the contour over which the integrals in Theorem~\ref{propTemmOscIntRepKBessel2} are defined.

%
%
%
\begin{proof}[Proof of Theorem~\ref{propTemmOscIntRepKBessel2}]

Changing variables $w \mapsto w - 2 \pi$, we have that \begin{align*}\int_{\frac 3 2 \pi}^{\frac 5 2 \pi} e^{-y \psi(u)} e^{-ru}e^{irw}\left(- \frac{\wrt u}{\wrt w}+i\right)~\wrt w =& e^{-2\pi t +i 2 \pi r} \int_{-\frac {\pi} 2}^{\frac \pi 2} e^{-y \psi(u)} e^{-ru}e^{irw}\left(- \frac{\wrt u}{\wrt w}+i\right)~\wrt w \\
\int_{\frac 3 2 \pi}^{\frac 5 2 \pi}e^{-y \psi(u)} e^{ru}e^{irw}\left(\frac{\wrt u}{\wrt w} +i\right)~\wrt w =& e^{-2\pi t +i 2 \pi r}\int_{-\frac {\pi}2}^{\frac \pi 2}e^{-y \psi(u)} e^{ru}e^{irw}\left(\frac{\wrt u}{\wrt w} +i\right)~\wrt w
  \end{align*}
  
  Now note that $w(u)$ is a bijection on $\eL^0$ and on $\eL^+$, and, thus, on each of these pieces, we may change our integrals with respect to $w$ to integrals with respect to $u$.
  
By calculus, we have that $\mu_-$ is the minimum value of $u$ on $\eL^+ \cup \eL^0$.  In particular, we have that $0<\mu_- < \mu$.  Hence, using the substitution theorem, we have that \begin{align*}\int_{\frac 1 2 \pi}^{\frac 3 2 \pi} e^{-y \psi(u)} e^{-ru}e^{irw}\left(- \frac{\wrt u}{\wrt w}+i\right)~\wrt w =& \int_{\mu_-}^{\mu} e^{-y \psi(u)} e^{-ru}e^{irw}\left( 1-i\frac{\wrt w}{\wrt u}\right)~\wrt u \\
\int_{\frac 1 2 \pi}^{\frac 3 2 \pi}e^{-y \psi(u)} e^{ru}e^{irw}\left(\frac{\wrt u}{\wrt w} +i\right)~\wrt w =& -\int_{\mu_-}^{\mu}e^{-y \psi(u)} e^{ru}e^{irw}\left(1 +i \frac{\wrt w}{\wrt u}\right)~\wrt u.
  \end{align*} 
  
 The desired result now follows.

\end{proof}

\begin{proof}[Proof of Theorem~\ref{thmSecondCaseOscillKBessel}]
 The integral representation in Theorem~\ref{propTemmOscIntRepKBessel2} is particularly suited to the saddle point method.  In each of the integrals from the proposition, we now change variables, letting $\rho :=u - \mu$, and estimate the integrals near $\rho = 0$.

First note that we have \[\sin w = \sin (w(\rho+ \mu))= \frac{\rho \cosh \mu + \sinh \mu}{\sinh \rho \cosh \mu + \cosh \rho \sinh \mu}.\]

Choose $\rho_0 >0$ small enough so we may apply the geometric series in the second equality of (\ref{eqnsinApproxOscilKBessel}).  Let $|\rho|\leq \rho_0.$  Using the Taylor series for $\sinh x$ and $\cosh x$, we have \begin{align}\label{eqnsinApproxOscilKBessel}
 \sin w =&  \frac{\rho \cosh \mu + \sinh \mu}{\sinh \mu \left(1 + \frac{\cosh \mu}{\sinh \mu}\rho + \frac{\rho^2}2 +\frac{\cosh \mu}{\sinh \mu} \frac{\rho^3}{6}+O(\rho^4)\right)} \\ \nonumber= &\left(\rho  \frac{\cosh \mu}{\sinh \mu}+1 \right)\left[ 1- \left(\frac{\cosh \mu}{\sinh \mu}\rho + \frac{\rho^2}2 +\frac{\cosh \mu}{\sinh \mu} \frac{\rho^3}{6}+O(\rho^4) \right) \right. \\\nonumber &\left. + \left(\frac{\cosh \mu}{\sinh \mu}\rho + \frac{\rho^2}2 +\frac{\cosh \mu}{\sinh \mu} \frac{\rho^3}{6}+O(\rho^4) \right)^2  \right. \\ \nonumber&\left.- \left(\frac{\cosh \mu}{\sinh \mu}\rho + \frac{\rho^2}2 +\frac{\cosh \mu}{\sinh \mu} \frac{\rho^3}{6}+O(\rho^4) \right)^3+O(\rho^4)\right]
 \\ \nonumber = &1 - \frac 1 2 \rho^2 + \left(-\frac{\cosh^2 \mu}{6 \sinh^2 \mu} + \frac{\cosh \mu}{2\sinh \mu} \right) \rho^3+O(\rho^4).\end{align}  Note that $\sinh \mu \neq 0$ because $\mu >0$.
 
 Now choose $\rho_1 >0$ small enough so we may apply the Taylor series for the function $\sqrt{1+x}$ in the third equality of (\ref{eqncosApproxOscilKBessel}).  Let us further restrict $\rho$ by requiring that $|\rho|\leq \min(\rho_0, \rho_1)$ holds.  We have \begin{align}\label{eqncosApproxOscilKBessel} \cos w &= \cos (w(\rho+ \mu))= \pm \sqrt{1 - \sin^2 w} \\\nonumber&= \pm|\rho| \sqrt{1 + \left(\frac{\cosh^2 \mu}{3 \sinh^2 \mu} - \frac{\cosh \mu}{\sinh \mu}  \right)\rho +O(\rho^2)} 
 \\\nonumber&= \pm \left(|\rho| +\frac 1 2 \left(\frac{\cosh^2 \mu}{3 \sinh^2 \mu} - \frac{\cosh \mu}{\sinh \mu}  \right)|\rho|\rho +O(\rho^3)\right). 
  \end{align}
  
  Now choose $\rho_2 >0$ small enough so that \begin{align}\label{eqnPosEstOscilKBessel}
 \left(|\rho| +\frac 1 2 \left(\frac{\cosh^2 \mu}{3 \sinh^2 \mu} - \frac{\cosh \mu}{\sinh \mu}  \right)|\rho|\rho +O(\rho^3)\right) \geq0.  \end{align} Here the implied constant is the same as the implied constant in (\ref{eqncosApproxOscilKBessel}).   Let again restrict $\rho$ by requiring that $|\rho|\leq \min(\rho_0, \rho_1, \rho_2)$ holds. 
  
For clarity, let us define the following: \begin{align*}
w^0(u) := &w(u) \textrm{ on }  \eL^0 \\
w^+(u) := &w(u) \textrm{ on }  \eL^+ \textrm{ for } 0 \leq w(u) \leq 3\pi/2 \end{align*} and let \begin{align*}
\psi^0(u) &:= \cosh u \cos w^0 +w^0 \cosh \mu  \\ \psi^+(u) &:= \cosh u \cos w^+ +w^+ \cosh \mu \end{align*} denote $\psi(u)$ on on $\eL^0$ and on $\eL^+$ corresponding to $0\leq w \leq3\pi/2$, respectively.  Note that on $\eL^+$, for small enough $\delta>0$, we have\[  \begin{cases} \cos \left(w^+(\delta + \mu)\right)< 0 \text { if } \delta <0  \\   \cos \left(w^+(\delta + \mu)\right)> 0 \text { if } \delta >0 \end{cases},\] and, thus, we have \[\cos w^+ = \cos (w^+(\rho+ \mu))= \rho +\frac 1 2 \left(\frac{\cosh^2 \mu}{3 \sinh^2 \mu} - \frac{\cosh \mu}{\sinh \mu}  \right)\rho^2 +O(\rho^3).\]
  
  On $\eL^0$, for small enough $\delta>0$, we have $\cos \left(w^0\right)>0$ if $\delta <0$, and, thus, we have \[\cos w^0 = \cos (w^0(\rho+ \mu)) = -\rho -\frac 1 2 \left(\frac{\cosh^2 \mu}{3 \sinh^2 \mu} - \frac{\cosh \mu}{\sinh \mu}  \right)\rho^2 -O(\rho^3).\]  Note that the implied constants in the expressions for $\cos w^+$ and $\cos w^0$ are the same.  Note that $\rho \leq 0$ on $\eL^0$.

  Now choose $\rho_3 >0$ small enough so we may apply the Taylor series for the function $\arccos(x)$ in the expressions for $\cos w^+$ and $\cos w^0$.  Let us again restrict $\rho$ by requiring that $|\rho|\leq \min(\rho_0, \rho_1, \rho_2, \rho_3)$ holds.  
 We have \begin{align}\label{eqnwApproxOscilKBessel} w^+ & = w^+(\rho+ \mu)=\frac \pi 2 -  \rho -\frac 1 2 \left(\frac{\cosh^2 \mu}{3 \sinh^2 \mu} - \frac{\cosh \mu}{\sinh \mu}  \right)\rho^2 +O(\rho^3) \\\nonumber w^0 &=w^0(\rho+ \mu)=\frac \pi 2 +  \rho +\frac 1 2 \left(\frac{\cosh^2 \mu}{3 \sinh^2 \mu} - \frac{\cosh \mu}{\sinh \mu}  \right)\rho^2 +O(\rho^3). 
  \end{align}
  
  Consequently, using the above and the Taylor series for $\cosh(x)$ and $\sinh(x)$, we have that \begin{align*}\psi^+(\rho+\mu) = & \cosh \rho \cosh \mu \cos w^+ + \sinh \rho \sinh \mu \cos w^+ + w^+ \cosh \mu  
  \\ =& \cosh \mu \left( \rho +\frac 1 2 \left(\frac{\cosh^2 \mu}{3 \sinh^2 \mu} - \frac{\cosh \mu}{\sinh \mu}  \right)\rho^2 \right)+ (\sinh \mu) p^2 \\& + \cosh \mu\left( \frac \pi 2 -  \rho -\frac 1 2 \left(\frac{\cosh^2 \mu}{3 \sinh^2 \mu} - \frac{\cosh \mu}{\sinh \mu}  \right)\rho^2\right)+O(\rho^3) 
  \\=& \frac \pi 2 \cosh \mu+(\sinh \mu) p^2 +O(\rho^3)
  \end{align*} and, analogously, \begin{align}\psi^0(\rho+\mu) = \frac \pi 2 \cosh \mu-(\sinh \mu) p^2 +O(\rho^3).
\end{align}

Also, choose $\rho_4 >0$ small enough so that \begin{align}\label{eqnWApproxOscilKBesselBndForA1} \left | -\frac 1 2 \left(\frac{\cosh^2 \mu}{3 \sinh^2 \mu} - \frac{\cosh \mu}{\sinh \mu}  \right)\rho^2 +O(\rho^3) \right| \leq \frac{|\rho|} 2. \end{align}

Finally, we have \begin{align*} \frac{\wrt {}}{\wrt \rho}\left(w^+(\rho +\mu)\right) &= -1 +O(\rho) \\ \frac{\wrt {}}{\wrt \rho}\left(w^0(\rho +\mu)\right) &= 1 +O(\rho).
  \end{align*}

  Let us now consider the integrals in (\ref{eqnTemmOscIntRepKBessel2}).  Applying the change of variables $\rho = u - \mu$, we have that  \begin{align*}
\int_{\mu_-}^{\infty} &e^{-y \psi^+(u)} e^{-ru}e^{irw^+}\left(1-i \frac{\wrt w^+}{\wrt u}\right) ~\wrt u \\ &=\int_{\mu_--\mu}^{\infty} e^{-y \psi^+(\rho+\mu)} e^{-r(\rho+\mu)}e^{irw^+(\rho+\mu)}\left(1-i \frac{\wrt {}}{\wrt \rho}\left(w^+(\rho +\mu)\right)\right) ~\wrt \rho  \end{align*}  

\begin{align*}\int_{\mu_-}^{\infty} &e^{-y \psi^+(u)} e^{ru}e^{irw^+}\left(1+i \frac{\wrt w^+}{\wrt u}\right) ~\wrt u \\ &=\int_{\mu_--\mu}^{\infty} e^{-y \psi^+(\rho+\mu)} e^{r(\rho+\mu)}e^{irw^+(\rho+\mu)}\left(1+i \frac{\wrt {}}{\wrt \rho}\left(w^+(\rho +\mu)\right)\right) ~\wrt \rho  
 \end{align*}

Pick $1/3 < \delta < 1/2$ and let $y_0^{-\delta} = \min(\rho_0, \rho_1, \rho_2, \rho_3, \rho_4, \mu - \mu_-, 1/2)$.  Pick $y_1 >0$ such that \begin{align}\label{eqnLargeCosKBessel}
 \frac{\left(y^{1-\delta} +\frac 1 2 \left(\frac{\cosh^2 \mu}{3 \sinh^2 \mu} - \frac{\cosh \mu}{\sinh \mu}  \right)y^{1-2\delta} +O(y^{1-3\delta})\right)\sinh \mu}{2}-|r| \geq 0 
 \\ \nonumber \left|\frac 1 2 \left(\frac{\cosh^2 \mu}{3 \sinh^2 \mu} - \frac{\cosh \mu}{\sinh \mu}  \right)y^{1-2\delta} +O(y^{1-3\delta}) \right|\leq \frac{y^{1-\delta}}{2}\end{align} hold for all $y \geq  y_1$.  Here, the implied constants are the same as that for $\cos w^+$ above.  We now estimate the integrals in a small neighborhood of the saddle point $R^+_0$ on $\eL^+$ determined by $|\rho| < y^{-\delta}$ where $y \geq \max(y_0, y_1)$.  Using the estimates above, we simplify in manner similar that in Section~\ref{subsecFirstCaseMonoKBessel} to obtain  \begin{align*}
\int_{-y^{-\delta}}^{y^{-\delta}} &e^{-y \psi^+(\rho+\mu)} e^{-r(\rho+\mu)}e^{irw^+(\rho+\mu)}\left(1-i \frac{\wrt {}}{\wrt \rho}\left(w^+(\rho +\mu)\right)\right) ~\wrt \rho  \\ =& e^{-y \frac \pi 2 \cosh \mu -r \mu +i r \frac \pi 2}\left[1+i+O\left(\max(y^{-\delta}, y^{1-3\delta})\right) \right] \int_{-y^{-\delta}}^{y^{-\delta}}  e^{-y (\sinh\mu) \rho^2} ~\wrt \rho  
\\ =& \sqrt{\frac{\pi}{y \sinh \mu}}e^{-y \frac \pi 2 \cosh \mu -r \mu +i r \frac \pi 2}\left[1+i+O\left(\max(y^{-\delta}, y^{1-3\delta})\right) \right]\end{align*}  where in the last equality we have changed variables $\rho^2 \mapsto y (\sinh\mu) \rho^2$ and applied the standard bounds for $\textrm{erfc}(x)$.

Likewise, we have \begin{align*}\int_{-y^{-\delta}}^{y^{-\delta}} &e^{-y \psi^+(\rho+\mu)} e^{r(\rho+\mu)}e^{irw^+(\rho+\mu)}\left(1+i \frac{\wrt {}}{\wrt \rho}\left(w^+(\rho +\mu)\right)\right) ~\wrt \rho \\=& \sqrt{\frac{\pi}{y \sinh \mu}}e^{-y \frac \pi 2 \cosh \mu +r \mu +i r \frac \pi 2}\left[1-i+O\left(\max(y^{-\delta}, y^{1-3\delta})\right) \right].
  \end{align*}

We now compute the dominant term, namely the following:   \begin{align} \label{eqnDomTermOscilCaseKBessel}\frac 1 2 e^{i \chi}&\sqrt{\frac{\pi}{y \sinh \mu}}e^{-y \frac \pi 2 \cosh \mu -r \mu +i r \frac \pi 2}\left(1+i \right) +  \frac 1 2 e^{-i\chi}\sqrt{\frac{\pi}{y \sinh \mu}}e^{-y \frac \pi 2 \cosh \mu +r \mu +i r \frac \pi 2}\left(1-i\right)
\\\nonumber= & \sqrt{\frac{\pi}{y \sinh \mu}}e^{-y \frac \pi 2 \cosh \mu +i r \frac \pi 2} \left(\cosh(i \chi - r \mu) +i \sinh(i \chi - r \mu)\right) 
\\\nonumber= & \sqrt{\frac{\pi}{y \sinh \mu}}e^{-y \frac \pi 2 \cosh \mu +i r \frac \pi 2} \left[\cosh(r \mu)\left( \cos \chi - \sin \chi\right) -i \sinh(r \mu)\left( \cos \chi + \sin \chi\right)\right]
\\\nonumber= & \sqrt{\frac{2\pi}{y \sinh \mu}}e^{-y \frac \pi 2 \cosh \mu +i r \frac \pi 2} \left[\cosh(r \mu) \sin\left(\frac \pi 4 - \chi\right)  -i \sinh(r \mu)\cos\left(\frac \pi 4 - \chi\right)\right].
  \end{align}  Here, the last equality follows from the elementary observation
  \begin{align*}
\cos(\chi) &= \sin\left(\frac \pi 4 + \frac \pi 4 - \chi\right) =  \frac{\sqrt 2} 2 \cos\left(\frac \pi 4 - \chi\right) + \frac{\sqrt 2} 2 \sin\left(\frac \pi 4 - \chi\right) 
\\\sin(\chi) &= \cos\left(\frac \pi 4 + \frac \pi 4 - \chi\right) =  \frac{\sqrt 2} 2 \cos\left(\frac \pi 4 - \chi\right) - \frac{\sqrt 2} 2 \sin\left(\frac \pi 4 - \chi\right).
\end{align*}   

To prove that (\ref{eqnDomTermOscilCaseKBessel}) is the dominant term, we now show that the rest is negligible.  As in Section~\ref{subsecFirstCaseMonoKBessel}, we shall pick a more suitable contour and use the Cauchy-Goursat theorem.  Let us consider the integral \begin{align*}A:=\frac 1 2 e^{-i\chi}\int_{\mu_-}^{\mu - y^{-\delta}} &e^{-y \psi^+(u)} e^{ru}e^{irw^+}\left(1+i \frac{\wrt w^+}{\wrt u}\right) ~\wrt u \\ &=\frac 1 2 e^{-i\chi}\int_{\mu_--\mu}^{- y^{-\delta}} e^{-y \psi^+(\rho+\mu)} e^{r(\rho+\mu)}e^{irw^+(\rho+\mu)}\left(1+i \frac{\wrt {}}{\wrt \rho}\left(w^+(\rho +\mu)\right)\right) ~\wrt \rho  
 \end{align*} and replace the piece of the path of steepest descent that the integral is over with the contours $\ell_1 \cup \ell_2$ where\begin{align*} \ell_1 &:= \left\{\left(\mu -y^{-\delta}- u, w^+(\mu - y^{-\delta})\right):  0\leq u\leq \mu -y^{-\delta}-\mu_-\right\}\\\ell_2 &:= \left\{\left(\mu_-, w^+(\mu - y^{-\delta})+\theta\right):  0\leq \theta \leq 3\pi/2 - w^+(\mu - y^{-\delta})\right\}.
  \end{align*}  The integral over $\ell_1$ is \begin{align*} A_1:=\frac 1 2 \int_{\mu - y^{-\delta} }^{\mu_-} e^{-y \left[\cosh\left(u+iw^+(\mu - y^{-\delta})\right) -i\left(u+iw^+(\mu - y^{-\delta})\right) \cosh \mu \right]} e^{r\left(u+iw^+(\mu - y^{-\delta})\right)} ~\wrt u,  
\end{align*}  where the integrand comes from (\ref{IntRepKBessel3}).  Changing variables, $\rho = u - \mu$, we have  \begin{align*} A_1=\frac {e^{-yw^+(\mu - y^{-\delta}) \cosh \mu+r\mu + irw^+(\mu - y^{-\delta})}} 2 \int_{ -y^{-\delta} }^{\mu_--\mu} &e^{-y \left[\cosh(\rho+\mu) \cos \left(w^+(\mu - y^{-\delta})\right)  \right]}  \\ & \times e^{-y \left[i\sinh(\rho+\mu) \sin \left(w^+(\mu - y^{-\delta})\right)-i (\rho+\mu) \cosh \mu \right]} e^{r\rho} ~\wrt \rho.
\end{align*}  Thus, as $-y^{-\delta} \geq \mu_- - \mu$, we have \begin{align*}
 |A_1| \leq & \frac {e^{-yw^+(\mu - y^{-\delta}) \cosh \mu+r\mu }} 2 \int_{\mu_--\mu}^{ -y^{-\delta} } e^{-y \left(\cosh(\rho+\mu) \cos \left(w^+(\mu - y^{-\delta})\right)  \right)} e^{r\rho} ~\wrt \rho 
 \\ \leq &  \frac {e^{r\mu-y\left[\cosh(\mu - y^{-\delta})  \cos \left(w^+(\mu - y^{-\delta})\right)+w^+(\mu - y^{-\delta}) \cosh \mu\right]}} {2} \int_{\mu_--\mu}^{ -y^{-\delta} }  e^{r\rho} ~\wrt \rho
 \\ \leq &  \frac {e^{r\mu-y \frac \pi 2 \cosh \mu-y\left(y^{-\delta} -\frac 1 2 \left(\frac{\cosh^2 \mu}{3 \sinh^2 \mu} - \frac{\cosh \mu}{\sinh \mu}  \right)y^{-2\delta} +O(y^{-3\delta})\right)  \left(\cosh \mu - \cosh(\mu - y^{-\delta}) \right)}} {2} \int_{\mu_--\mu}^{ -y^{-\delta} }  e^{r\rho} ~\wrt \rho.\end{align*} Here the second inequality follows because $\cos \left(w^+(\mu - y^{-\delta})\right) <0$ and the third from the approximations we computed above and the fact that (\ref{eqnPosEstOscilKBessel}) holds. 
 
 By the mean value theorem, there exists $\mu - y^{-\delta} < \widetilde{\mu}< \mu$ such that $\cosh(\mu - y^{-\delta}) = \cosh \mu  - y^{-\delta}\sinh \widetilde{\mu}$.  Since $0<\mu_-\leq \mu - y^{-\delta}$ holds, we have that $\sinh\widetilde{\mu} >0$.  Consequently, we have the following upper bound: \begin{align*}
|A_1| \leq&  \frac {e^{r\mu-y \frac \pi 2 \cosh \mu-\left(y^{1-2\delta} -\frac 1 2 \left(\frac{\cosh^2 \mu}{3 \sinh^2 \mu} - \frac{\cosh \mu}{\sinh \mu}  \right)y^{1-3\delta} +O(y^{1-4\delta})\right)\sinh\widetilde{\mu}}} {2} \int_{\mu_--\mu}^{ -y^{-\delta} }  e^{r\rho} ~\wrt \rho
\\ \leq & \begin{cases}  \frac 1 2 e^{r\mu-y \frac \pi 2 \cosh \mu- y^{1-2\delta}\sinh\widetilde{\mu} } (1+O(y^{1-3\delta}))O(1 + y^{-\delta}) &\text { if } r \neq 0 \\  \frac 1 2 e^{-y \frac \pi 2 \cosh \mu- y^{1-2\delta}\sinh\widetilde{\mu}}(1+O(y^{1-3\delta}))&\text { if } r = 0\end{cases} \end{align*}  As $1-2\delta >0$, this shows that $A_1$ is negligible compared to the dominant term, as desired.


The integral over $\ell_2$ is \begin{align*} A_2:=\frac i 2 \int_{w^+(\mu - y^{-\delta})}^{\frac 3 2 \pi} e^{-y \left[\cosh\left(\mu_- +i \theta\right) -i\left(\mu_- +i \theta \right) \cosh \mu \right]} e^{r\left(\mu_- +i \theta\right)} ~\wrt \theta.  
\end{align*}  Using (\ref{eqnwApproxOscilKBessel}, \ref{eqnWApproxOscilKBesselBndForA1}), we have \begin{align*} |A_2| \leq & \frac 1 2 e^{r (\mu_-)} \int_{\frac \pi 2+ \frac{y^{-\delta}}{2}}^{\frac 3 2 \pi} e^{-y\left(\cosh (\mu_-) \cos \theta + \theta \cosh \mu \right)}~\wrt \theta
\\ \leq &\frac 1 2 e^{r (\mu_-)} \int_{\frac \pi 2+ \frac{y^{-\delta}}{2}}^{\frac 3 2 \pi} e^{-y\left(\cosh (\mu_-) \left(\frac \pi 2 - \theta \right) + \theta \cosh \mu \right)}~\wrt \theta
\\ = &\frac 1 2 e^{r (\mu_-)-y\cosh (\mu_-)\frac \pi 2 } \int_{\frac \pi 2+ \frac{y^{-\delta}}{2}}^{\frac 3 2 \pi} e^{- y\theta\left(\cosh \mu - \cosh (\mu_-) \right) }~\wrt \theta 
\\ = & \frac {e^{r (\mu_-)-y\frac \pi 2 \cosh \mu-  \frac{y^{1-\delta}}{2} \left( \cosh \mu - \cosh (\mu_-) \right)}} {2y \left( \cosh \mu - \cosh (\mu_-) \right)}\left(1 - e^{\left(-y \pi +\frac{y^{1-\delta}}{2}\right)\left(\cosh \mu - \cosh (\mu_-) \right)} \right),
  \end{align*}  where the second inequality follows from the fact that $\pi/2 - \theta \leq \cos \theta$ for all $\theta \geq \pi/2$.  This shows that $A_2$ and, thus, $A$ are negligible compared to the dominant term, as desired.
 
 Let us consider the integral \begin{align*} B:=  \frac 1 2 e^{-i\chi} \int_{-\frac {\pi}2}^{\frac 3 2 \pi }e^{-y \psi(u)} e^{ru}e^{irw}\left(\frac{\wrt u}{\wrt w} +i\right)~\wrt w
  \end{align*} and replace the piece of the path of steepest descent with the contour \[\ell_3:= \left\{(\mu_-, \theta):  -\frac\pi 2 \leq \theta \leq \frac 3 2\pi \right\}\] to obtain \[B =  \frac i 2  \int_{-\frac {\pi}2}^{\frac 3 2 \pi }e^{-y \left[\cosh\left(\mu_- +i \theta\right) -i\left(\mu_- +i \theta \right) \cosh \mu \right]} e^{r\left(\mu_- +i \theta\right)} ~\wrt \theta\] where the integrand comes from (\ref{IntRepKBessel3}). Thus, we have \begin{align*} |B| \leq & \frac 1 2 e^{r (\mu_-)+ y \cosh(\mu_-)} \int_{-\frac \pi 2}^{\frac 3 2 \pi} e^{-y\theta \cosh \mu}~\wrt \theta = \frac{e^{r (\mu_-)+ y \cosh(\mu_-)}}{2 y \cosh \mu}\left(e^{y \frac \pi 2 \cosh \mu} -  e^{-y \frac {3\pi} 2 \cosh \mu}\right)
  \end{align*} and that \begin{align*} \left|\frac{e^{-2 \pi t + i 2 \pi r}}{1-e^{-2 \pi t + i 2 \pi r}}B\right| \leq \frac{e^{r (\mu_-)+ y \cosh(\mu_-)}}{ y \cosh \mu}\left(e^{-y \frac {3\pi} 2 \cosh \mu} -  e^{-y \frac {7\pi} 2 \cosh \mu}\right)
  \end{align*}  is negligible compared to the dominant term, as desired.
  
Let $U_0 = \mu+1/2$.  Let us now consider the integral \begin{align*}C:=\frac 1 2 e^{-i\chi}\int_{\mu +y^{-\delta}}^{\infty}&e^{-y \psi^+(u)} e^{ru}e^{irw^+}\left(1+i \frac{\wrt w^+}{\wrt u}\right) ~\wrt u \\ &=\frac 1 2 e^{-i\chi}\int_{ y^{-\delta}}^\infty e^{-y \psi^+(\rho+\mu)} e^{r(\rho+\mu)}e^{irw^+(\rho+\mu)}\left(1+i \frac{\wrt {}}{\wrt \rho}\left(w^+(\rho +\mu)\right)\right) ~\wrt \rho  
 \end{align*} and replace the piece of the path of steepest descent that the integral is over with the family of contours \begin{align*}
\ell_4:=\ell_4(U):=&\{(u, w^+(\mu +y^{-\delta})): \mu +y^{-\delta} \leq u\leq U \} 
\\ \ell_5:=\ell_5(U):=&\{(U, w^+(\mu +y^{-\delta})-\theta): 0 \leq \theta  \leq w^+(\mu +y^{-\delta})-w^+(U) \}
 \end{align*} for any $U \geq U_0$. The integral over $\ell_4$ is \begin{align*}C_1:= C_1(U):=\frac 1 2 \int_{\mu + y^{-\delta} }^{U} e^{-y \left[\cosh\left(u+iw^+(\mu + y^{-\delta})\right) -i\left(u+iw^+(\mu + y^{-\delta})\right) \cosh \mu \right]} e^{r\left(u+iw^+(\mu + y^{-\delta})\right)} ~\wrt u,  
\end{align*}  where the integrand comes from (\ref{IntRepKBessel3}).

Changing variables, $\rho = u - \mu$, we have, for all $U \geq U_0$, \begin{align*}
 |C_1| \leq & \frac {e^{-yw^+(\mu + y^{-\delta}) \cosh \mu+r\mu }} 2 \int_{y^{-\delta} }^\infty e^{-y \cosh(\rho+\mu) \cos \left(w^+(\mu + y^{-\delta})\right)} e^{|r|\rho} ~\wrt \rho
\\  \leq & \frac {e^{-yw^+(\mu + y^{-\delta}) \cosh \mu+r\mu }} 2 \int_{y^{-\delta} }^\infty e^{-y \left(\cosh \rho \cosh \mu + \sinh \rho \sinh \mu \right) \cos \left(w^+(\mu + y^{-\delta})  \right)} e^{|r|\rho} ~\wrt \rho 
\\  \leq & \frac {e^{-yw^+(\mu + y^{-\delta}) \cosh \mu - y \cosh \mu \cos \left(w^+(\mu + y^{-\delta})  \right) +r\mu }} 2 \int_{y^{-\delta} }^\infty e^{-y \rho \sinh \mu  \cos \left(w^+(\mu + y^{-\delta})  \right)} e^{|r|\rho} ~\wrt \rho
\\  \leq & \frac {e^{-y \frac \pi 2 \cosh \mu +O(y^{1-3\delta)} +r\mu }} 2 \int_{y^{-\delta} }^\infty e^{-\rho\left(y^{1-\delta} +\frac 1 2 \left(\frac{\cosh^2 \mu}{3 \sinh^2 \mu} - \frac{\cosh \mu}{\sinh \mu}  \right)y^{1-2\delta} +O(y^{1-3\delta})\right)\sinh \mu} e^{|r|\rho} ~\wrt \rho 
\\  \leq & \frac {e^{-y \frac \pi 2 \cosh \mu +O(y^{1-3\delta)} +r\mu }} 2 \int_{y^{-\delta} }^\infty e^{-\frac 1 2 \rho\left(y^{1-\delta} +\frac 1 2 \left(\frac{\cosh^2 \mu}{3 \sinh^2 \mu} - \frac{\cosh \mu}{\sinh \mu}  \right)y^{1-2\delta} +O(y^{1-3\delta})\right)\sinh \mu} ~\wrt \rho
\\  \leq & \frac {4 e^{-y \frac \pi 2 \cosh \mu - \frac 1 4 y^{1-2 \delta} \sinh \mu +r\mu }} {y^{1-\delta} \sinh \mu}  \left(1+O(y^{1-3\delta)}\right)
 \end{align*} where the third inequality follows because $0 < w^+(\mu + y^{-\delta}) < \pi/2$, $\cosh \rho \geq 1$, and $\sinh \rho \geq \rho$ for $\rho >0$, where the fourth inequality follows from our approximations above, and where the fifth and sixth inequalities follow from (\ref{eqnLargeCosKBessel}).  This shows the integral over $C_1(U)$ is negligible compared to the dominant term for all $U\geq U_0$, as desired.
 
 The integral over $\ell_5$ is  \begin{align*}C_2:= C_2(U):=\frac i 2 \int_{w^+(\mu +y^{-\delta}) }^{w^+(U)} e^{-y \left[\cosh\left(U+i\theta\right) -i\left(U+i\theta\right) \cosh \mu \right]} e^{r\left(U+i\theta\right)} ~\wrt \theta,  
\end{align*}  where the integrand comes from (\ref{IntRepKBessel3}).  We have \begin{align*}|C_2|  \leq & \frac 1 2 e^{r U} \int_{0}^{w^+(\mu +y^{-\delta}) } e^{-y\left(\cosh U \cos \theta + \theta \cosh \mu \right)}~\wrt \theta
\\ \leq &\frac 1 2 e^{r U}w^+(\mu +y^{-\delta})e^{-y\cosh U \cos \left( w^+(\mu +y^{-\delta})\right)}
\\ = & \frac 1 2 e^{r U}\left(\frac \pi 2 - O(y^{-\delta}) \right)e^{-\cosh U\left( y^{1-\delta} +\frac 1 2 \left(\frac{\cosh^2 \mu}{3 \sinh^2 \mu} - \frac{\cosh \mu}{\sinh \mu}  \right)y^{1-2\delta} +O(y^{1-3\delta})\right)}
  \end{align*} where the second inequality follows from the observation that $0<\cos \left( w^+(\mu +y^{-\delta})\right) \leq \cos \theta$ for all $\theta \in [0, w^+(\mu +y^{-\delta})]$.  This shows that the integral over $C_2(U)$ is negligible compared to the dominant term for all large enough $U$, as desired.  Consequently, the integral over $C$ is negligible compared to the dominant term, as desired.
  
For the remaining integrals, we note that we have changed variables $u \mapsto -u$ the integrals over $\eL^-$ to arrive at the integral representation in Proposition~\ref{propTemmOscIntRepKBessel1}.  We have already estimated above the integral around $(-\mu, \pi/2)$ along $\eL^-$.  Using the Cauchy-Goursat theorem, we now replace the remaining with integrals along the following contours \begin{align*} \ell_1^- &:= \left\{\left(-\mu +y^{-\delta}+ u, w(-\mu + y^{-\delta})\right):  0\leq u\leq \mu -y^{-\delta}-\mu_-\right\}
\\\ell_2^- &:= \left\{\left(-\mu_-, w(-\mu + y^{-\delta})+\theta\right):  0\leq \theta \leq 3\pi/2 - w(-\mu + y^{-\delta})\right\}
\\\ell_3^-&:= \left\{(-\mu_-, \theta):  -\frac\pi 2 \leq \theta \leq \frac 3 2\pi \right\}
\\\ell_4^-&:=\ell_4^-(U):=\{(u, w(-\mu -y^{-\delta})): -U \leq u \leq -\mu -y^{-\delta}  \} 
\\ \ell_5^-&:=\ell_5^-(U):=\{(-U, w(-U)-\theta): 0 \leq \theta  \leq -w(-\mu -y^{-\delta})+w(-U) \}.  
  \end{align*}  The integrals over these contours are, respectively \begin{align*}
  A^-_1&:=\frac 1 2 \int_{-\mu + y^{-\delta} }^{-\mu_-} e^{-y \left[\cosh\left(u+iw(-\mu + y^{-\delta})\right) -i\left(u+iw(-\mu + y^{-\delta})\right) \cosh \mu \right]} e^{r\left(u+iw(-\mu + y^{-\delta})\right)} ~\wrt u
  \\A^-_2&:=\frac i 2 \int_{w(-\mu + y^{-\delta})}^{\frac 3 2 \pi} e^{-y \left[\cosh\left(-\mu_- +i \theta\right) -i\left(-\mu_- +i \theta \right) \cosh \mu \right]} e^{r\left(-\mu_- +i \theta\right)} ~\wrt \theta
\\B^- &:=  \frac i 2  \int_{-\frac {\pi}2}^{\frac 3 2 \pi }e^{-y \left[\cosh\left(-\mu_- +i \theta\right) -i\left(-\mu_- +i \theta \right) \cosh \mu \right]} e^{r\left(-\mu_- +i \theta\right)} ~\wrt \theta
\\C^-_1&:= C^-_1(U):=\frac 1 2 \int_{-U}^{-\mu - y^{-\delta} } e^{-y \left[\cosh\left(u+iw(-\mu - y^{-\delta})\right) -i\left(u+iw(-\mu - y^{-\delta})\right) \cosh \mu \right]} e^{r\left(u+iw(-\mu - y^{-\delta})\right)} ~\wrt u
\\C^-_2&:= C^-_2(U):=\frac i 2 \int_{w(-U)}^{w(-\mu -y^{-\delta}) } e^{-y \left[\cosh\left(-U+i\theta\right) -i\left(-U+i\theta\right) \cosh \mu \right]} e^{r\left(-U+i\theta\right)} ~\wrt \theta
  \end{align*} where each of the integrands comes from (\ref{IntRepKBessel3}).  Changing variables $u \mapsto -u$, using the observations that $w(x)$ and $\cosh x$ are both even functions, and applying the analogous proofs we used to show that $A_1, A_2, B, C_1, C_2$, respectively, are negligible compared to the dominant term, we have that $A^-_1, A^-_2, B^-, C^-_1, C^-_2$, respectively, are negligible compared to the dominant term, as desired.  This proves the theorem.

\end{proof}


\begin{thebibliography}{99}

\bibitem{AS64} M.~Abramowitz, I.~A.~Stegun, ``Handbook of mathematical functions with formulas, graphs, and mathematical tables.''
National Bureau of Standards Applied Mathematics Series, {\bf 55}, U.S. Government Printing Office, Washington, D.C. 1964.\\
%

\bibitem{Bal66} C.~B.~Balogh, {\em Asymptotic expansions of the modified Bessel function of the third kind of imaginary order.} SIAM J. Appl. Math. {\bf 15} (1967), 1315--1323. \\
\bibitem{Boc1892} M.~B\^ocher, {\em On some applications of Bessel's functions with pure imaginary index,} Ann. of Math. {\bf6} (1892), no. 6, 137--160.\\ 
\bibitem{BST} A.~R.~Booker, A.~Str\"ombergsson, and H.~Then, {\em Bounds and algorithms for the K-Bessel function of imaginary order.} LMS J. Comput. Math. {\bf 16} (2013), 78--108.\\

\bibitem{Con07} J.~T.~Conway, {\em Inductance Calculations for Noncoaxial Coils Using Bessel Functions,} IEEE Transactions on Magnetics, vol. {\bf43}, no. 3, (March 2007), 1023--1034.\\ 

\bibitem{Cop} E.~T.~Copson, ``Asymptotic expansions.'' Reprint of the 1965 original. Cambridge Tracts in Mathematics, {\bf 55.} Cambridge University Press, Cambridge, 2004.\\

\bibitem{DRS10} R.~Diaz, W.~J.~Rice, and D.~L~Stokes, {\em Fourier-Bessel reconstruction of helical assemblies,} Methods Enzymol, {\bf482} (2010),131--165.\\


\bibitem{Erd56} A.~Erd\'elyi, ``Asymptotic expansions.'' Dover Publications, Inc., New York, 1956. \\

\bibitem{EMOT} A.~Erd\'elyi, W.~Magnus, F.~Oberhettinger, and F.~G.~Tricomi. ``Higher transcendental functions.'' Vol. II. McGraw-Hill Book Company, Inc., New York, 1953.\\

\bibitem{Iwa02} H.~Iwaniec, ``Spectral methods of automorphic forms,'' second edition, Graduate Studies in Mathematics, {\bf 53}, American Mathematical Society, Providence, RI; Revista Matem\'atica Iberoamericana, Madrid, 2002.\\

\bibitem{Kor02} B.~G.~Korenev, ``Bessel functions and their applications,'' Translated from the Russian by E. V. Pankratiev. Analytical Methods and Special Functions, {\bf 8}, Taylor \& Francis Group, London, 2002.\\
\bibitem{Mac1899} H.~M.~MacDonald, {\em Zeroes of the Bessel Functions,} Proc. Lond. Math. Soc. {\bf 30} (1898/99), 165--179. \\

\bibitem{DLMF} {\em NIST Digital Library of Mathematical Functions.} \url{http://dlmf.nist.gov/}, Release 1.1.0 of 2020-12-15. F.~ W.~J. Olver, A.~B.~Olde Daalhuis, D.~W. Lozier, B.~I.~Schneider, R.~F.~Boisvert, C.~W.~Clark, B.~R.~Miller, B.~V.~Saunders, H.~S.~Cohl, and M.~A.~McClain, eds.\\
\bibitem{Olv} F.~W.~J. Olver, ``Asymptotics and special functions.'' Computer Science and Applied Mathematics. Academic Press [A subsidiary of Harcourt Brace Jovanovich, Publishers], New York-London, 1974.\\

\bibitem{Rob90} C.~Robert, {\em Modified Bessel functions and their applications in probability and statistics,} Statist. Probab. Lett. {\bf9} (1990), no. 2, 155--161.\\


\bibitem{SF75} E.~O.~Steinborn and E.~Filter, {\em Translations of fields represented by spherical-harmonic expansions for molecular calculations: Translations of reduced Bessel functions, Slater-types-orbitals, and other functions,} Theoret. Chim. Acta {\bf38}, (1975)  273--281.\\
\bibitem{St} A.~ Str\"ombergsson, {\em On the uniform equidistribution of long closed horocycles,} Duke Math. J. {\bf 123} (2004), no. 3, 507--547. \\
\bibitem{Tem} N.~M.~Temme, {\em Steepest descent paths for integrals defining the modified Bessel functions of imaginary order.}  Methods Appl. Anal. {\bf 1} (1994), no. 1, 14--24. \\

\bibitem{Ts18} J.~Tseng, {\em Eisenstein series and an asymptotic for the $K$-Bessel function,}  Ramanujan J. {\bf 56} (2021), no. 1, 323--345.\\
\bibitem{Wa} G.~N.~Watson, ``A Treatise on the Theory of Bessel Functions,'' Cambridge University Press, Cambridge, England, 1944.\\
\bibitem{Zha10} P.~Zhao, {\em Quantum variance of Maass-Hecke cusp forms,} Comm. Math. Phys. {\bf 297} (2010), no. 2, 475--514.\\ 

\end{thebibliography}
\end{document}